\documentclass[11pt]{amsart}
\usepackage[hmargin=2.5cm,vmargin=2.5cm]{geometry}
\usepackage{amsfonts}
\usepackage{amsthm}
\usepackage{cancel}
\usepackage[english]{algorithm2e}
\usepackage{amsmath}
\usepackage{amssymb}
\usepackage{hyperref}
\usepackage[T1]{fontenc}
\newtheorem{Theorem}{Theorem}[section]
\newtheorem{theorem}[Theorem]{Theorem}
\newtheorem{definition}[Theorem]{Definition}
\newtheorem{remark}[Theorem]{Remark}
\newtheorem{example}[Theorem]{Example}

\newtheorem{proposition}[Theorem]{Proposition}

\title{Representations of weighted Rota-Baxter Jacobi-Jordan algebras}
\author[J. Anitch\'eou, S. Attan ]
{Jules Anitch\'eou and Sylvain Attan 
} 
\address{Jules Anitch\'eou \newline
	Institut de Math\'ematiques et de Sciences Physiques, Universit\'{e} d'Abomey-Calavi
	01 BP 613-Oganla, Porto-Novo, B\'{e}nin}
\email{jules.anitcheou@imsp-uac.org}
\address{Sylvain Attan \newline
	D\'{e}partement de Math\'{e}matiques, Universit\'{e} d'Abomey-Calavi
	01 BP 4521, Cotonou 01, B\'{e}nin}
\email{syltane2010@yahoo.fr}
\begin{document}
	\maketitle
	\begin{abstract} 
		Weighted Rota-Baxter Jacobi-Jordan algebras and their representations are studied. Moreover, we consider weighted Rota-Baxter paired operators  that are related to weighted Rota-Baxter Jacobi-Jordan algebras together with their representations. Finally, we define suitable  cohomology for weighted Rota-Baxter Jacobi-Jordan algebras in low  degrees.	
	\end{abstract}
	
		
	\section{Introduction}
	A while ago, a new class of algebras appeared in the literature: the so Jacobi-Jordan algebras. These are  commutative non-associative algebras  satisfying the Jacobi identity. It turns out that these algebras seen as strange cousins of Lie algebras, are a special class of Jordan
	nilalgebras and are also related to Bernstein-Jordan algebras.They were  first appeared in \cite{Momo6} and since then a lot of works are done on this subject, see for example \cite{Momo4,Momo5,pzu}.

	A Rota-Baxter operator of weight $0$ on an algebra $A$ is a linear map $\mathcal{I}: A\rightarrow A$ satisfying the so-called Rota-Baxter identity: $\mathcal{I}(x)\mathcal{I}(y)=\mathcal{I}(\mathcal{I}(x)y+x\mathcal{I}(y)),$ for all $x,y\in A$. Integration by parts is an example of a Rota-Baxter operator of weight $0$. Thus, These operators can be regarded as an algebraic abstraction
	representing the integral operator. A Rota-Baxter algebra is an associative algebra equipped with a Rota-Baxter operator of weight $0$.  These operators  first introduced by Baxter in his study of the fluctuation theory \cite{Toto4}, were further developed by Rota \cite{Toto23} where some identities as well as some applications of Baxter algebra arising in probability and combinatorial theory were given. An important fact is that, Rota-Baxter operators have a connection with combinatorics of shuffle algebras \cite{gk}, Yang-Baxter equations \cite{Toto34}, dendriform algebras \cite{Toto34, Toto11}, renormalizations in quantum field theory \cite{Toto6}, multiple
zeta values in number theory \cite{Toto14} and splitting of algebraic operads \cite{bbgn}. The reader can see  studies of Cartier \cite{Toto33} and Atkinson \cite{fva} for further properties of Rota-Baxter operators as well as those of \cite{bak} where these operators on Lie algebras were first considered by Kuperscmidt in the study of classical $r$-matrices.
More generally, weighted Rota-Baxter operators (also called Rota-Baxter operators with arbitrary weight)  related with tridendriform algebras \cite{Toto10}, post-Lie algebras and modified Yang-Baxter equations \cite{Toto2}, weighted infinitesimal bialgebras and weighted Yang-Baxter equations \cite{Toto31},  were studied in \cite{Toto2, Toto3}. Recently, the cohomology of Rota-Baxter operators of weight $1$ on Lie algebras and Lie groups is studied \cite{Toto16}. Furthermore, Das \cite{Toto7} developed a similar theory for Rota-Baxter associative algebras of weight zero. Next, Wang and Zhou \cite{Toto28}, Das \cite{Toto8} studied Rota-Baxter algebras of arbitrary weight by different methods respectively. Precisely in \cite{Toto8}, the authors considered weighted Rota-Baxter associative algebras and define cohomology of them with coefficients in a suitable Rota-Baxter bimodule. When considering the cohomology with coefficients in itself, it governs the simultaneous deformations of algebras and weighted Rota-Baxter operators. Later, Das \cite{das} studied cohomology of weighted Rota-Baxter Lie
algebras and Rota-Baxter paired operators using the method of \cite{Toto28}. This may be a transformation from weight zero to arbitrary weight, and the successful transformation is entirely due to
Rota-Baxter Lie groups \cite{gls}.  Up to now, there was very few study about Rota-Baxter operators  on Jacobi-Jordan algebras. Thus it is time to study these operators on Jacobi-Jordan algebras. In a concrete way, we study a representation theory and cohomology theory in low degrees  of Rota-Baxter Jacobi-Jordan algebras of arbitrary weight. To do that, we will apply the approach of \cite{das} to weighted Rota-Baxter Jacobi-Jordan algebras. More precisely, we first consider representations of weighted Rota-Baxter Jacobi-Jordan algebras and provide various constructions of them. Next, we define the cohomology of a weighted Rota-Baxter Jacobi-Jordan algebra with coefficients in a representation. This cohomology as some existing works on the theme, is obtained as a byproduct of the standard zigzag cohomology of the underlying Jacobi-Jordan algebra and the cohomology of the underlying weighted Rota-Baxter operator.

The paper is organized as follows. In Section \ref{s2} we consider weighted Rota-Baxter Jacobi-Jordan algebras and their representations. We also give several new constructions of representations on these algebraic structures. In section \ref{s3}, given a Jacobi-Jordan algebra and a representation on it, we introduce weighted Rota-Baxter paired operators that are related to weighted Rota-Baxter Jacobi-Jordan algebras together with their representations and we give a characterization of this concept. Observe that the terminology of Rota-Baxter paired operators   appeared first in \cite{Toto32} where the authors introduced Rota-Baxter paired modules in the associative context and was taken up in \cite{das} in the case of Lie algebras. In section \ref{s4}, we defined the cohomology in low degrees of a weighted Rota-Baxter Jacobi-Jordan algebra with coefficients in its representation and we compute some cohomology spaces as applications.

Throughout this paper, let $\mathbb{K}$ be a field of characteristic $0.$ Except specially stated, vector spaces are $\mathbb{K}$-vector spaces and all tensor products are taken over $\mathbb{K}$.
	
	\section{Representatons of weighted Rota-Baxter Jacobi-Jordan algebras}\label{s2}
	In this section, we consider a weighted Rota-Baxter Jacobi-Jordan and introduce their representations. We also provide various examples and new constructions. Let $\lambda \in \mathbb{K}$ be a fixed scalar unless specified otherwise.
		\begin{definition}
		 A $\lambda$- weighted Rota-Baxter operator on an algebra $(A,\ast)$ is a linear map $\mathcal{I}:A \longrightarrow A$ satisfying
			\begin{eqnarray}
				\mathcal{I}(x)\ast\mathcal{I}(y)=\mathcal{I}(\mathcal{I}(x)\ast y+x\ast \mathcal{I}(y)+\lambda (x\ast y)), \forall x,y \in A. \label{ro}
			\end{eqnarray}	
	\end{definition}
\begin{example}
	Let' s take $A=C(\mathbb{R})$, the algebra of continuous real-valued functions on $\mathbb{R}$ and define $\mathcal{I}: A\rightarrow A$ by 
	$$\mathcal{I}(f)(x):=\int_{0}^{x}f(t)dt.$$
	 Then $\mathcal{I}$ is a $0$- weighted Rota-Baxter operator on the algebra $A$.
\end{example}
	\begin{definition} 
		\begin{enumerate}
		\item[(i)] A Jacobi-Jordan algebra is an algebra $(A,\ast)$ satisfying
		\begin{eqnarray}
			&& x\ast y=y\ast x, \forall x,y\in A,\nonumber\\
			&&(x\ast y)\ast z+ (y\ast z)\ast x+(z\ast x)\ast y; \forall x,y,z\in A.\label{jji}
		\end{eqnarray}
	\item [(ii)] A $\lambda$- weighted Rota-Baxter Jacobi-Jordan algebra is a pair $(A,\mathcal{I})$ consisting of a  Jacobi-Jordan algebra $A$ together with a $\lambda$- weighted Rota-Baxter operator on it.
		\item [(iii)]	
	Let $(A,\mathcal{I})$ and $(A',\mathcal{I}')$ be two $\lambda$-weighted Rota-Baxter Jacobi-Jordan algebras. A morphism from $(A,\mathcal{I})$ to $(A',\mathcal{I}')$ is  algebras morphism $\phi:A\rightarrow A'$ satisfying  $\phi\circ\mathcal{I}=\mathcal{I}'\circ\phi$. It is called an isomorphism if $\phi$ is so.
\end{enumerate}	
	\end{definition}
	\begin{example} $\label{E}$
		\begin{enumerate} 
			\item [(i)] Let $A$ be a Jacobi-Jordan algebra then, the pair $(A,id_{A})$ is a $(-1)$-weighted Rota-Baxter Jacobi-Jordan algebra.
			\item [(ii)] Let $(A,\mathcal{I})$ be a $\lambda$-weighted Rota-Baxter Jacobi-Jordan algebra. Then for all $\mu\in\mathbb{K}$, the pair $(A,\mu\mathcal{I})$ is a $\mu\lambda$-weighted Rota-Baxter algebra. 
			\item [(iii)] Let $(A,\mathcal{I})$  be  a $\lambda$-weighted Rota-Baxter Jacobi-Jordan algebra  and  $\psi\in Aut(A)$ be an automorphism  of the Jacobi-Jordan algebra $A$. Thn, the pair  $(A,\psi^{-1}\circ\mathcal{I}\circ \psi)$ is a $\lambda$-weighted Rota-Baxter Jacobi-Jordan algebra.
			\item [(iv)] Let $(A,\mathcal{I})$ be a $\lambda$- weighted Rota-Baxter Jacobi-Jordan algebra. Then the pair $(A,-\lambda id_{A}-\mathcal{I})$ is also a $\lambda$- weighted Rota-Baxter Jacobi-Jordan algebra.
	\end{enumerate}
\end{example}
\begin{example}\label{F}
	Let $(A,\ast)$ be the two-dimensional Jacobi-Jordan algebra where the only nonzero product in a given basis $\{e_{1}, e_{2}\}$  is $e_{1}\ast e_{1}=e_{2}$ (see \cite{Momo6}).
	The linear map $\mathcal{I}: A\rightarrow A$ defined by 
	$\mathcal{I}(e_1)=a_1e_1+a_2e_2,\ \mathcal{I}(e_2)=b_1e_1+b_2e_2$, is a 
	Rota-Baxter operator of weight $\lambda\in\mathbb{R}$ on $(A,\ast)$ if and only if
	\begin{eqnarray}
		\mathcal{I}(e_i)\ast \mathcal{I}(e_j)=
		\mathcal{I}(\mathcal{I}(e_i)\ast e_j+e_i\ast \mathcal{I}(e_j)+\lambda e_i\ast e_j),   \mbox{ $  \forall\,\ 1\leq i,j\leq 2.$} \label{eq1}
	\end{eqnarray}
The equation (\ref{eq1}) holds if and only if
	$$(2a_1+\lambda)b_1=a_1^2-(2a_1+\lambda)b_2=a_1b_1-b_1b_2=b_1=0.$$
	\begin{itemize}
	\item [$\diamond$] If $\lambda=0$, then the pair $(A,\mathcal{I})$ is a $0$-weighted Rota-Baxter Jacobi-Jordan algebra where 
	\begin{eqnarray}
		\mathcal{I}\in\left\{\begin{pmatrix}
			0&0	\\
			a_2&b_2	
		\end{pmatrix} \mbox{, $ a_2, b_2  \in\mathbb{R}$ } \right\}
		\bigcup
		\left\{\begin{pmatrix}
			2a_1&0	\\
			a_2&a_1		
		\end{pmatrix} \mbox{, $ a_1, a_2  \in\mathbb{R}, a_1\neq 0$ } \right\}.\nonumber				
	\end{eqnarray}	
\item[$\diamond$] If $\lambda\neq 0$ and $a\neq \frac{\lambda}{2}$, then the pair $(A,\mathcal{I})$ is  a $\lambda$-weighted Rota-Baxter Jacobi-Jordan algebra where 
\begin{eqnarray}
	\mathcal{I}\in\left\{\begin{pmatrix}
		a_1&0	\\
		a_2&\frac{a_1^2}{2a_1+\lambda}	
	\end{pmatrix} \mbox{, $ a_1, a_2  \in\mathbb{R},\ a\neq \frac{\lambda}{2}$ } \right\}.			
\end{eqnarray}	
	\end{itemize}
\end{example}
\begin{example}\label{G}
Let $A:=Span\{e_{1}, e_{2}, e_{3}, e_{4}\}$ be the $4$-dimensional Jacobi-Jordan algebra over $\mathbb{R}$ defined by: $e_{1}\ast e_{1}=e_{2}$ (see \cite{Momo4}). Then the linear map $\mathcal{I}:A\rightarrow A$ given in the basis $\{e_{1}, e_{2}, e_{3}, e_{4}\}$ by 
\begin{eqnarray}
	\mathcal{I}=		
	\begin{pmatrix}
		a_{1}&b_{1}&c_{1}&d_{1}\\
		a_{2}&b_{2}&c_{2}&d_{2}\\
		a_{3}&b_{3}&c_{3}&d_{3}\\
		a_{4}&b_{4}&c_{4}&d_{4}			
	\end{pmatrix}
	, a_{i},b_{i},c_{i},d_{i}\in \mathbb{R}	\nonumber
\end{eqnarray}
is a $\lambda$-weighted Rota-Baxter operator on 
$(A,\ast)$ if and only if
\begin{eqnarray}
	\mathcal{I}(e_{i}) \ast \mathcal{I}(e_{j})=\mathcal{I}\Big(\mathcal{I}(e_{i}) \ast e_{j} + e_{i} \ast \mathcal{I}(e_{j})+\lambda (e_{i}\ast e_{j})\Big),  \forall  1 \leqslant i \leqslant j \leqslant 4,	\nonumber
\end{eqnarray}
that is  
\begin{equation}
	\left\{\begin{array}{ll}
		b_{1}=c_{1}=d_{1}=0, \\
		a^{2}_{1}=(2a_{1}+\lambda)b_{2}\\
		(2a_{1}+\lambda)(b_{3}-b_{4})=0.	
	\end{array} \right.\nonumber
\end{equation}
\begin{enumerate}
\item [$\diamond$] If $\lambda=0$, we get a
$0$-weighted Rota-Baxter Jacobi-Jordan algebra
$(A, \mathcal{I})$ where
 \begin{eqnarray}
	\mathcal{I}\in\left\{	
	\begin{pmatrix}
		0&0&0&0\\
		a_{2}&b_{2}&c_{2}&d_{2}\\
		a_{3}&b_{3}&c_{3}&d_{3}\\
		a_{4}&b_{4}&c_{4}&d_{4}			
	\end{pmatrix}
	, a_{i},b_{i},c_{i},d_{i}\in \mathbb{R}	\right\}
	\bigcup\left\{	
	\begin{pmatrix}
		a_{1}&0&0&0\\
		a_{2}&\frac{a^{2}_{1}}{2}&c_{2}&d_{2}\\
		a_{3}&b_{3}&c_{3}&d_{3}\\
		a_{4}&b_{3}&c_{4}&d_{4}			
	\end{pmatrix}
	, a_{i},b_{i},c_{i},d_{i}\in \mathbb{R}, a_1\neq 0		\right\}.	\nonumber
\end{eqnarray}
\item [$\diamond$] If $\lambda\neq 0$ and $a_1\neq\frac{\lambda}{2}$, we get a 
 $\lambda$-weighted Rota-Baxter Jacobi-Jordan algebra 
$(A, \mathcal{I})$ where
\begin{eqnarray}
\mathcal{I}\in\left\{	
\begin{pmatrix}
	a_{1}&0&0&0\\
	a_{2}&\frac{a^{2}_{1}}{2a_{1}+\lambda}&c_{2}&d_{2}\\
	a_{3}&b_{3}&c_{3}&d_{3}\\
	a_{4}&b_{3}&c_{4}&d_{4}			
\end{pmatrix}
, a_{i},b_{i},c_{i},d_{i}\in \mathbb{R}, a_1\neq\frac{\lambda}{2}	\right\}.	\nonumber
\end{eqnarray}
\end{enumerate}
\end{example}
\begin{example}\label{H}
	Let $A:=Span\{e_{1}, e_{2}, e_{3}, e_{4}\}$ be the $4$-dimensional Jacobi-Jordan algebra over $\mathbb{R}$ defined by: $e_{1}\ast e_{1}=e_{2},  e_{3}\ast e_{3}=e_{2}$ (see \cite{Momo4}). Then the linear map $\mathcal{I}:A\rightarrow A$ given in the basis $\{e_{1}, e_{2}, e_{3}, e_{4}\}$  by 
	\begin{eqnarray}
		\mathcal{I}=		
		\begin{pmatrix}
			a_{1}&b_{1}&c_{1}&d_{1}\\
			a_{2}&b_{2}&c_{2}&d_{2}\\
			a_{3}&b_{3}&c_{3}&d_{3}\\
			a_{4}&b_{4}&c_{4}&d_{4}			
		\end{pmatrix}
		, a_{i},b_{i},c_{i},d_{i}\in \mathbb{R} \nonumber	
	\end{eqnarray}
	is a $0$-weighted Rota-Baxter operator on 
	$(A,\ast)$ if and only if
	\begin{eqnarray}
		\mathcal{I}(e_{i}) \ast \mathcal{I}(e_{j})=\mathcal{I}\Big(\mathcal{I}(e_{i}) \ast e_{j} + e_{i} \ast \mathcal{I}(e_{j})\Big),  \forall  1 \leqslant i \leqslant j \leqslant 4,	\nonumber
	\end{eqnarray}
	that is  
	\begin{equation}
		\left\{\begin{array}{ll}
			b_{1}=b_{3}=d_{1}=d_{3}=0, \\
			a_{1}b_{4}= c_{3}b_{4}=(c_{1}+a_{3})b_{4}=0 \\
			a^{2}_{1} + a^{2}_{3} -2a_{1}b_{2}=c^{2}_{1}+c^{2}_{3} - 2c_{3} b_{2}=0 \\
			(a_{1}c_{1}+a_{3}c_{3})-b_{2}(c_{1}+a_{3})=0.
		\end{array} \right.\nonumber
	\end{equation}		
	Note that we do not classify all Rota-Baxter operators of the considered Jacobi-Jordan algebra. Hence, the pair $(A,\mathcal{I})$ is  a $0$-weighted Rota-Baxter Jacobi-Jordan algebras where 
	\begin{eqnarray}
		&&	\mathcal{I}\in\left\{\begin{pmatrix}
			0&0&0&0	\\
			a_2&0&c_2&d_2\\
			0&0&0&0\\
			a_4&b_4&c_4&d_4	
		\end{pmatrix} \mbox{, $ a_i, b_i,c_i,d_i \in\mathbb{R}$ } \right\}
		\bigcup
		\left\{\begin{pmatrix}
			0&0&0&0	\\
			a_2&b_2&c_2&d_2\\
			0&0&0&0\\
			a_4&b_4&c_4&d_4	
		\end{pmatrix} \mbox{, $ a_i, b_i,c_i,d_i \in\mathbb{R}, b_2b_4\neq 0$ } \right\}.\nonumber				
	\end{eqnarray}
\end{example}
Let recall the following definition from \cite{Momo5}.
	\begin{definition}
		Let $A$ be a Jacobi-Jordan algebra. A representation of $A$ \cite{Momo5} is a vector space $V$ with a linear map  $\rho : A\rightarrow gl(V)$ satisfying \begin{eqnarray}
				\rho(x\ast y)=-\rho(x)\rho(y)-\rho(y)\rho(x), \forall x,y \in A.\label{rpJJa}
			\end{eqnarray}
		A representation of $A$ as above is denoted by $(V,\rho)$ or simply by $V$ if no confusion arises. Observe that any Jacobi-Jordan algebra $A$ is a representation of itself with the action map  \\$\rho : A\rightarrow gl(V)$ defined by \begin{eqnarray}
				\rho(x)y:=x \ast y, \forall x,y \in A. \nonumber
			\end{eqnarray}
			Such representation is called the regular or adjoint representation of $(A,\ast).$ 
	\end{definition}
	\begin{definition}
		Let $(A,\mathcal{I})$ be a $\lambda$-weighted Rota-Baxter Jacobi-Jordan algebra. A representation of it is a triple $(V,\rho,\mathcal{T})$ in which $(V,\rho)$ is a representation of the Jacobi-Jordan algebra $A$ and $\mathcal{T} : V\rightarrow V$ is linear map such that \begin{eqnarray}
			\rho(\mathcal{I}x)(\mathcal{T}u) = \mathcal{T}\Big(\rho(\mathcal{I}x)u+\rho(x)(\mathcal{T}u)+\lambda\rho(x)u\Big), \forall x\in A, u\in V. \label{rpRJJa}
		\end{eqnarray} 
	A representation of $(A,\mathcal{I})$ as above is denoted simply by $(V,\mathcal{T})$.
	\end{definition}	
	\begin{example}	
		\begin{enumerate}
			\item [(i)]
			Any $\lambda$- weighted Rota-Baxter Jacobi-Jordan algebra $(A,\mathcal{I})$ is a representation of itself. Such representation is called the adjoint representation.
			\item [(ii)]
			Let $V$ be a representation of a Jacobi-Jordan algebra $A$. Then the pair $(V,id_{V})$ is a representation of the $(-1)$-weighted  Rota-Baxter Jacobi-Jordan algebra $(A,id_{A})$.
			\item [(iii)] Let $(V,\mathcal{T})$ be a representation of a $\lambda$-weighted Rota-Baxter Jacobi-Jordan algebra $(A,\mathcal{I})$. Then for any scalar $\mu\in \mathbb{K}$, the pair $(V,\mu\mathcal{T})$ is a representation of the $\mu\lambda$- weighted Rota-Baxter Jacobi-Jordan algebra $(A,\mu\mathcal{I}).$ 
		\end{enumerate}	
	\end{example}
	\begin{proposition}
		Let $(V,\mathcal{T})$ be a representation of a $\lambda$-weighted Rota-Baxter Jacobi-Jordan algebra $(A,\mathcal{I}).$ Then $(V,-\lambda id_{V}-\mathcal{T})$ is a representation of the $\lambda$-weighted Rota-Baxter Jacobi-Jordan algebra $(A,-\lambda id_{A}-\mathcal{I})$. 
	\end{proposition}
	\begin{proof} Pick $ x \in A, u\in V,$ then, we have:
		\begin{eqnarray}					
			\rho((-\lambda id_{A}-\mathcal{I})x)(-\lambda id_{V}-\mathcal{T})(u)&=&
			\lambda^{2}\rho(x)u+\lambda\rho(\mathcal{I}x)u+\lambda\rho(x)\mathcal{T}u+\rho(\mathcal{I}x)\mathcal{T}u \nonumber\\
			&=&\lambda^{2}\rho(x)u+\lambda\rho(\mathcal{I}x)u+\lambda\rho(x)\mathcal{T}u  \nonumber\\
			&+&\mathcal{T}\Big(\rho(\mathcal{I}x)u+\rho(x)\mathcal{T}u+\lambda\rho(x)u\Big) \label{BA} \mbox{( by (\ref{rpRJJa}) )}		
		\end{eqnarray}
		Next, we obtain
		\begin{eqnarray}
			(-\lambda id_{V}-\mathcal{T})\Big( \rho((-\lambda id_{A}-\mathcal{I})x)u+\rho(x)(-\lambda id_{V}-\mathcal{T})u+\lambda\rho(x)u\Big) \nonumber \\
			=\lambda^{2}\rho(x)u+\lambda\mathcal{T}(\rho(x)u)+\lambda\rho(\mathcal{I}x)u+\mathcal{T}(\rho(\mathcal{I}x)u)+\cancel{\lambda^{2}\rho(x)u}+\cancel{\lambda\mathcal{T}(\rho(x)u)}\nonumber \\
			+\lambda\rho(x)\mathcal{T}u 
			+\mathcal{T}(\rho(x)\mathcal{T}u)-\cancel{\lambda^{2}\rho(x)u}-\cancel{\lambda\mathcal{T}(\rho(x)u}). \label{DA}
		\end{eqnarray}		
		The expressions  in (\ref{BA}) and (\ref{DA}) are the same. Thus, the pair 
		$(-\lambda id_{V}-\mathcal{T})$ is a representation of $(A,-\lambda id_{A}-\mathcal{I})$.	
	\end{proof}
	\begin{proposition}
		Let $\{(V_{i},\mathcal{T}_{i})\}_{i\in I\subset \mathbb{N}}$ be a family of representations of a $\lambda$-weighted  Rota-Baxter Jacobi-Jordan algebra $(A,\mathcal{I})$. Then the pair $(\oplus _{i\in I}V_{i},\oplus _{i \in I}\mathcal{T}_{i})$ is a representation of $(A,\mathcal{I}).$
	\end{proposition}
	\begin{proof}
		Let $\rho_{i}:A\rightarrow gl(V_{i})$ denote the action of the Jacobi-Jordan algebra $A$ on the representation $V_{i}.$ Then it follows that $\rho:A\rightarrow gl(\oplus_{i\in I}V_{i})$, $\rho(x)(u_{i})_{i \in I}= (\rho_{i}(x)u_{i})_{i\in I}$ is a representation of $(A,\ast)$ on $\oplus_{i \in I}V_{i}.$ In addition, let $ x \in A$ and $(u_{i})_{i\in I} \oplus_{i\in I}V_{i}$, then we have \begin{eqnarray}
			\rho(\mathcal{I}x)(\oplus_{i\in I}{\mathcal{T}_{i}})(u_{i})_{i\in I}&=&\Big(\rho_{i}(\mathcal{I}x)\mathcal{T}_{i}(u_{i})\Big)_{i\in I} \nonumber \\
			&=& \Big(\mathcal{T}_{i}(\rho_{i}(\mathcal{I}x)(u_{i})+\rho_{i}(x)\mathcal{T}_{i}(u_{i})+\lambda\rho_{i}(x)u_{i})\Big)_{i\in I}  \mbox{ ( by (\ref{rpRJJa}) )}\nonumber\\
			&=&(\oplus_{i\in I}\mathcal{T}_{i})\Big(\rho(\mathcal{I}x)(u_{i})_{i\in I}+\rho(x)(\oplus_{i\in I}\mathcal{T}_{i}(u_{i}))_{i\in I}+\lambda\rho(x)(u_{i})_{i \in I}\Big).\nonumber
		\end{eqnarray}	
		Hence the result follows.
	\end{proof} 
Given  a representation $(V,\rho)$  of a Jacobi-Jordan algebra $A$, there is a representation of $A$ on the space $gl(V)$ with the action given by \begin{eqnarray}
		\hat{\rho}: A\rightarrow gl(gl(V)),
		(\hat{\rho}(x)f)u=-f(\rho(x)u), \forall (x,f,u) \in A\times gl(V)\times V. \nonumber	
	\end{eqnarray}
	With this representation, we obtain the following result:
	\begin{proposition}
		Let $(V,\mathcal{T})$ be a representation of a $\lambda$-weighted Rota-Baxter Jacobi-Jordan $(A,\mathcal{I})$. Then the pair $(gl(V),\hat{\mathcal{T}})$ is also a representation of $(A,\mathcal{I})$, where 
		\begin{eqnarray}
			\hat{\mathcal{T}}&:& gl(V)\rightarrow gl(V) ,
			\hat{\mathcal{T}}(f)(u)=-\lambda f(u)-f(\mathcal{T}(u)), \forall (f,u)\in gl(V)\times V. \nonumber
		\end{eqnarray}
	\end{proposition}
	\begin{proof}
		Let $ x\in A, u\in V$ and $f\in gl(V)$, then we compute
		\begin{eqnarray}
			\Big(\hat{\rho}(\mathcal{I}x)\hat{\mathcal{T}}(f)\Big)u=-\hat{\mathcal{T}}(f)(\rho(\mathcal{I}x)u) 
			=\lambda f(\rho(\mathcal{I}x)u)+	f(\mathcal{T}(\rho(\mathcal{I}x)u)). \label{Ja}	 	
		\end{eqnarray}
	Next, we also  have
\begin{eqnarray}
	&&	\Big((\hat{\rho}(\mathcal{I}x)f+\hat{\rho}(x)\hat{\mathcal{T}}(f)+\lambda\hat{\rho}(x)f\Big)u\nonumber\\
		&&=-\lambda\Big(\hat{\mathcal{T}}(\hat{\rho}(\mathcal{I}x)f+\hat{\rho}(x)\hat{\mathcal{T}}(f)+\lambda\hat{\rho}(x)f\Big)u
		-\Big(\hat{\rho}(\mathcal{I}x)f+
		\hat{\rho}(x)\hat{\mathcal{T}}(f)+\lambda\hat{\rho}(x)f\Big)(\mathcal{T}u)\nonumber \\
		&&=\lambda\Big(f(\rho(\mathcal{I}x)u)+\hat{\mathcal{T}}(f)(\rho(x)u)+\lambda f(\rho(x)u)\Big)
		+\Big(f(\rho(\mathcal{I}x)(\mathcal{T}u))
		+\hat{\mathcal{T}}(f)(\rho(x)(\mathcal{T}u)) 
		+\lambda f(\rho(x)(\mathcal{T}u)\Big)\nonumber\\
		&&=\lambda\Big( f(\rho(\mathcal{I}x)u)-\cancel{\lambda f(\rho(x)u)}-\cancel{f(\mathcal{T}{\rho}(x)u)}+\cancel{\lambda f(\rho(x)u))}\Big) 
		+f\Big(\mathcal{T}(\rho(\mathcal{I}x)u)+\cancel{\mathcal{T}(\rho(x)\mathcal{T}u)}+\cancel{\mathcal{T}(\rho(x)u)}\nonumber\\
		&&-\cancel{\lambda\rho(x)(\mathcal{T}u)}-\cancel{\mathcal{T}(\rho(x)(\mathcal{T}u)}
		+\cancel{\lambda\rho(x)(\mathcal{T}u)}\Big)\nonumber\\
		&&=\lambda f(\rho(\mathcal{I}x)u)+f(\mathcal{T}(\rho(\mathcal{I}x)u)).\label{Jb}			
\end{eqnarray}
Hence, by  (\ref{Ja}) and (\ref{Jb}),  $(V,\hat{\mathcal{T}})$ is a representation of $(A, \mathcal{I})$.
	\end{proof}
As expected, let construct now a semidirect product in the context of $\lambda$-weight Rota-Baxter Jacobi-Jordan algebras.
	\begin{proposition}\label{sdpr}
		Let $(V,\mathcal{T})$ be a representation of a $\lambda$-weighted Rota-Baxter Jacobi-Jordan algebra $(A, \mathcal{I})$. Then $(A\oplus V,\mathcal{I}\oplus\mathcal{T})$ is a $\lambda$-weighted Rota-Baxter Jacobi-Jordan algebra, where the product on $A\oplus V$ is given by the semidirect product 
		\begin{eqnarray}
		(x,u)\ast_{\ltimes}(y,v):=\Big(x\ast y,\rho(x)v+\rho(y)u\Big), \forall x,y\in A, \forall u,v\in V.\label{spd1}
		\end{eqnarray} 
	\end{proposition}
	\begin{proof}
		The proof of the commutativity of $\ast_{\ltimes}$ is obvious. 
			Now, pick
			$x,y,z \in A, u,v,w \in V$,	then  using the Jacobi identity in $(A,\ast)$ and (\ref{rpJJa}) we have			 
			\begin{eqnarray}
				\begin{split}				
			\Big((x,u)\ast_{\ltimes}(y,v)\Big)\ast_{\ltimes}(z,w)+\Big((y,v)\ast_{\ltimes}(z,w)\Big)\ast_{\ltimes}(x,u)+\Big((z,w)\ast_{\ltimes}(x,u)\Big)\ast_{\ltimes}(y,v) \nonumber \\			
			=\Big(\Big((x\ast y)\ast z +(y\ast z)\ast x+(z\ast x)\ast y)\Big),			
			\Big(\rho(x\ast y)w+\rho(x)\rho(y)w+\rho(y)\rho(x)w\Big)\nonumber\\
			+\Big(\rho(y\ast z)u +\rho(y)\rho(z)u+\rho(z)\rho(y)u \Big) 
			+\Big(\rho(z\ast x)v+\rho(z)\rho(x)v+\rho(x)\rho(z)v\Big)\Big)=0.				
			\end{split}						
			\end{eqnarray}
Hence, the Jacobi identity in $A\oplus V$ holds. 	
Finally, to show that $\mathcal{I}\oplus\mathcal{T}$ is a $\lambda$-weight Rota-Baxter operator on the semi-direct product $\lambda$-weighted Rota-Baxter Jacobi-Jordan algebra, pick $\forall x,y\in A, u,v\in V$. Then, we observe first that 
			\begin{eqnarray}					
			&&	(\mathcal{I}\oplus\mathcal{T})(x,u)\ast_{\ltimes} (\mathcal{I}\oplus\mathcal{T})(y,v)=(\mathcal{I}x,\mathcal{T}u)\ast_{\ltimes}(\mathcal{I}y,\mathcal{T}v)
				=\Big(\mathcal{I}x\ast \mathcal{I}y,\rho(\mathcal{I}x)\mathcal{T}v+\rho(\mathcal{I}y)\mathcal{T}u\Big) \nonumber\\
				&&=\Big(\mathcal{I}\Big(\mathcal{I}x\ast y+x\ast \mathcal{I}y+\lambda x\ast y\Big)
				,\mathcal{T}\Big(\rho(\mathcal{I}x)v+\rho(x)\mathcal{T}v+\lambda\rho(x)v\Big)
				\nonumber\\
				&&+\mathcal{T}\Big(\rho(\mathcal{I}y)u+\rho(y)\mathcal{T}u+\lambda\rho(y)u\Big)\Big) \mbox{ ( by (\ref{ro}) and (\ref{rpRJJa}) ).}\label{sd1}
				\end{eqnarray}
	Next, we have
		\begin{eqnarray}
		(\mathcal{I}\oplus\mathcal{T})(x,u)\ast_{\ltimes}(y,v)=\Big(\mathcal{I}x\ast y,\rho(\mathcal{I}x)v+\rho(y)\mathcal{T}u\Big) \label{Gu} \\
		(x,u)\ast_{\ltimes}(\mathcal{I}\oplus\mathcal{T})(y,v)=\Big(x\ast\mathcal{I}y,\rho(x)\mathcal{T}v+\rho(\mathcal{T}y)u\Big) \label{Go}\\	
		\lambda\Big((x,u)\ast_{\ltimes}(y,v)\Big)=\lambda\Big(x\ast y,\rho(x)v+\rho(y)u\Big)\label{Ga}	
		\end{eqnarray}
	i.e.,
	\begin{eqnarray}
	&&\Big((\mathcal{I}\oplus\mathcal{T})(x,u)\ast_{\ltimes}(y,v)+(x,u)\ast_{\ltimes}(\mathcal{I}\oplus\mathcal{T})(y,v)+\lambda\Big((x,u)\ast_{\ltimes}(y,v)\Big)\Big)\nonumber\\
	&&=\Big(\mathcal{I}x\ast y+x\ast\mathcal{I}y+\lambda x\ast y,
	\rho(\mathcal{I}x)v+\rho(x)\mathcal{T}v+\lambda\rho(x)v+\rho(\mathcal{I}y)u+\rho(y)\mathcal{T}u+\lambda\rho(y)u\Big).\label{sd2}
	\end{eqnarray}
Hence, it follows from (\ref{sd1}) and (\ref{sd2}) that
\begin{eqnarray}			
&&(\mathcal{I}\oplus\mathcal{T})\Big((\mathcal{I}\oplus\mathcal{T})(x,u)\ast_{\ltimes}(y,v)+(x,u)\ast_{\ltimes}(\mathcal{I}\oplus\mathcal{T})(y,v)+\lambda(x,u)\ast_{\ltimes}(y,v)\Big)\nonumber\\
&&=(\mathcal{I}\oplus\mathcal{T})(x,u)\ast_{\ltimes} (\mathcal{I}\oplus\mathcal{T})(y,v).	\nonumber
\end{eqnarray}			
	The $\lambda$-weighted Rota-Baxter Jacobi-Jordan algebra $(A\oplus V,\mathcal{I}\oplus\mathcal{T})$  is called the semidirect product of $(A,\mathcal{I})$ by $(V,\mathcal{T})$.
\end{proof}	
One can prove also the converse of the Proposition \ref{sdpr} and therefore, we have:
\begin{theorem}
	Let $(A,\mathcal{I})$ be a $\lambda$-weighted Rota-Baxter Jacobi-Jordan algebra. Let $V$ be a vector space and $\rho: A \longrightarrow gl(V),$ $\mathcal{T}: V\rightarrow V$ be two linear maps.
	Then $(V=(V, \rho),  \mathcal{T})$ is a representation of the $\lambda$-weighted Rota-Baxter Jacobi-Jordan algebra  $(A,\mathcal{I})$ if and only if
	$(A\oplus V,\mathcal{I}\oplus\mathcal{T})$ is a  $\lambda$-weighted Rota-Baxter Jacobi-Jordan algebra, where $A\oplus V$ is equipped with the product given in (\ref{spd1}).
\end{theorem}
	\begin{proposition} $\label{Po}$
		Let $(A,\mathcal{I})$ be a $\lambda$-weighted Rota-Baxter Jacobi-Jordan algebra. 
		\begin{enumerate}
			\item [(i)] The pair $(A,\ast_{\mathcal{I}})$ is a Jacobi-Jordan algebra, where  
			\begin{eqnarray}
				x\ast_{\mathcal{I}}y:= \mathcal{I}(x)\ast y+x \ast \mathcal{I}(y)+\lambda(x \ast y),\forall x,y\in A. \label{TO}		
			\end{eqnarray}
			We denote this algebra by $A_{\mathcal{I}}.$
			\item [(ii)] The pair $(A_{\mathcal{I}},\mathcal{I})$	is a $\lambda$- weighted Rota-Baxter Jacobi-Jordan algebra and the map \\ $\mathcal{I}:A_{\mathcal{I}}\longrightarrow A$ is a morphism of $\lambda$- weighted Rota-Baxter Jacobi-Jordan algebras.
		\end{enumerate}		 
	\end{proposition}
	\begin{proof}
	Let $ x,y,z \in A$, then observe first that
			\begin{eqnarray}
			&&	(x\ast_{\mathcal{I}}y)\ast_{\mathcal{I}}z=(\mathcal{I}x \ast \mathcal{I}y)\ast z+(\mathcal{I}x\ast y)\ast \mathcal{I}z+(x\ast\mathcal{I}y)\ast\mathcal{I}z \nonumber \\
				&+&\lambda(x\ast y)\ast\mathcal{I} z+(\mathcal{I}x\ast y)\ast z +\lambda(x\ast \mathcal{I} y)\ast z+ \lambda^{2}(x\ast y)\ast z. \nonumber	
			\end{eqnarray}					 
	Next, the Jacobi identity in $(A,\ast_{\mathcal{I}})$ is deduced from the one in $(A,\ast)$ as it follows:
		\begin{eqnarray}					
			&&(x\ast_{\mathcal{I}}y)\ast_{\mathcal{I}}z+(y\ast_{\mathcal{I}}z)\ast_{\mathcal{I}}x+(z\ast_{\mathcal{I}} x)\ast_{\mathcal{I}}y\nonumber\\
			&=&\Big((\mathcal{I}(x)\ast\mathcal{I}(y))\ast z+(\mathcal{I}(y)\ast z)\ast\mathcal{I} (x)+(z\ast\mathcal{I}(x))\ast \mathcal{I}(y)\Big)\nonumber\\
			&+&\Big((\mathcal{I}(x)\ast y)\ast \mathcal{I}(z)+(y\ast\mathcal{I} (z))\ast \mathcal{I}(x)+(\mathcal{I}(z)\ast \mathcal{I}(x))\ast y\Big)\nonumber\\
			&+& \Big((x\ast\mathcal{I}(y))\ast\mathcal{I}(z)+(\mathcal{I}(y)\ast\mathcal{I}(z))\ast x+(\mathcal{I}(z)\ast x)\ast\mathcal{I}(y)\Big) \nonumber \\
			&+& \lambda\Big((x\ast y)\ast\mathcal{I}(z)+(y\ast \mathcal{I}(z))\ast x+(\mathcal{I}(z)\ast x)\ast y)\Big)\nonumber\\
			&+& \lambda\Big((\mathcal{I}(x)\ast y)\ast z+(y\ast z)\ast \mathcal{I}(x)+(z\ast \mathcal{I}(x))\ast y\Big) \nonumber\\
			&+& \lambda^{2}\Big((x\ast y)\ast z+(y\ast z)\ast x+(z\ast x)\ast y\Big)=0. \nonumber		
		\end{eqnarray}
	Next, it is obvious that
	\begin{eqnarray}
		\mathcal{I}(x\ast_{\mathcal{I}}y)=\mathcal{I}(x)\ast\mathcal{I}(y)\label{mra}
	\end{eqnarray}
	and finally, we obtain
	 \begin{eqnarray}
			\begin{split}			
			\mathcal{I}(x)\ast_{\mathcal{I}}\mathcal{I}(y)&=\mathcal{I}^{2}(x)\ast\mathcal{I}(y)+\mathcal{I}(x)\ast\mathcal{I}^{2}(y)+\lambda\mathcal{I}(x)\ast\mathcal{I}(y) \\ \nonumber
			&=\mathcal{I}\Big(\mathcal{I}(x)\ast_{\mathcal{I}}y+x\ast\mathcal{I}(y)+\lambda x\ast_{\mathcal{I}}y\Big) \mbox{ (by (\ref{mra}) ).}
		\end{split}		
		\end{eqnarray}
	It follows that	
		$(A_{\mathcal{I}},\mathcal{I})$	is a $\lambda$- weighted Rota-Baxter Jacobi-Jordan algebra and  $\mathcal{I}$ is a morphism of $\lambda$- weighted Rota-Baxter Jacobi-Jordan algebras from $(A_{\mathcal{I}},\mathcal{I})$ to $(A,\mathcal{I}).$
		 	\end{proof}	
	\begin{theorem}\label{rwra}
		Let $(V,\mathcal{T})$ be a representation of a $\lambda$-weighted Rota-Baxter Jacobi-Jordan algebra $(A,\mathcal{I})$.  Define the map $\bar{\rho}:A\rightarrow gl(V)$ by 
		\begin{eqnarray}
			\bar{\rho}(x)(u):=\rho(\mathcal{I}x)(u)+\rho(x)(\mathcal{T}(u))+\lambda\rho(x)(u), \forall x\in A, u\in V.\nonumber			
		\end{eqnarray}
		\begin{enumerate}
			\item[(i)]
			$\bar{\rho}$ satisfies $\mathcal{T}(\bar{\rho}(x)u)=\rho(\mathcal{I}x)\mathcal{T}(u)$ $\label{TN}$,
			\item [(ii)] $(\bar{V}=(V,\bar{\rho}),\mathcal{T})$ is a representation of the $\lambda$- weighted Rota-Baxter Jacobi-Jordan algebra $(A_{\mathcal{I}},\mathcal{I}).$		
		\end{enumerate}
	\end{theorem}
	\begin{proof}
	Obviously, we obtain the part (i) from (\ref{rpRJJa}). Next, to prove the part (ii), pick $x,y\in A$ and $u\in V$, then, first we have 
		\begin{eqnarray}						
		&&	\bar{\rho}(x)\bar{\rho}(y)(u)+ \bar{\rho}(y)\bar{\rho}(x)(u)\nonumber\\
		&=&\rho(\mathcal{I}x)\bar{\rho}(y)u+\rho(x)\mathcal{T}(\bar{\rho}(y)u)+\lambda\rho(x)\bar{\rho}(y)u
			+\rho(\mathcal{I}y)\bar{\rho}(x)u+\rho(y)\mathcal{T}(\bar{\rho}(x)u)+\lambda\rho(y)\bar{\rho}(x)u \nonumber\\
			&=&\rho(\mathcal{I}x)\rho(\mathcal{I}y)u+\rho(\mathcal{I}x)\rho(y)\mathcal{T}(u)+\lambda\rho(\mathcal{I}x)\rho(y)u+\rho(x)\rho(\mathcal{I}y)\mathcal{T}(u)
			+\lambda\rho(x)\rho(\mathcal{I}y)u\nonumber\\
			&&+\lambda\rho(x)\rho(y)\mathcal{T}(u)+\lambda^{2}\rho(x)\rho(y)u 
			+\rho(\mathcal{I}y)\rho(\mathcal{I}x)u+\rho(\mathcal{I}y)\rho(x)\mathcal{T}(u)+\lambda\rho(\mathcal{I}y)\rho(x)u\nonumber\\
			&&+
			\rho(y)\rho(\mathcal{I}x)\mathcal{T}(u)+\lambda\rho(y)\rho(\mathcal{I}x)u+\lambda\rho(y)\rho(x)\mathcal{T}(u)+\lambda^{2}\rho(y)\rho(x)u \nonumber\\
			&=&-\rho(\mathcal{I}x\ast \mathcal{I}y)u-\rho(\mathcal{I}x\ast y+x\ast \mathcal{I}y+\lambda x\ast y)\mathcal{T}(u)-\lambda\rho(\mathcal{I}x\ast y+x\ast \mathcal{I}y+\lambda x\ast y) \mbox{( by (\ref{rpJJa}) )}\nonumber\\
			&=&-\rho(\mathcal{I}(x\ast_{\mathcal{I}}y))u-\rho((x\ast_{\mathcal{I}}y)\mathcal{T}(u)-\lambda\rho(x\ast_{\mathcal{I}}y)u 
			=-\rho(x\ast_{\mathcal{I}}y)u.\nonumber
		\end{eqnarray}
		It follows that $\bar{V}=(V,\bar{\rho})$ is a representation of the Jacobi-Jordan algebra $A_{\mathcal{I}}.$  Next, we have:
		\begin{eqnarray}					
			\bar{\rho}(\mathcal{I}x)\mathcal{T}(u) &=& \rho(\mathcal{I}^{2}(x))\mathcal{T}(u)+\rho(\mathcal{I}x)\mathcal{T}^{2}(u)+\lambda\rho(\mathcal{I}x)\mathcal{T}(u) \nonumber \\
			&=& \mathcal{T}\Big(\bar{\rho}(\mathcal{I}x)(u)+\bar{\rho}(x)\mathcal{T}(u)+\lambda\bar{\rho}(x)(u)\Big)	\mbox{ ( from the part (i) ).}\nonumber		
		\end{eqnarray}
		Hence, $(\bar{V},\mathcal{T})$ is a representation of $(A_{\mathcal{I}},\mathcal{I}).$	
	\end{proof}
	\begin{remark}
		When considering the adjoint representation $(A,\mathcal{I})$ of the $\lambda$- weighted Rota-Baxter Jacobi-Jordan algebra $(A,\mathcal{I})$, the representation above in Theorem \ref{rwra} is the adjoint representation of the $\lambda$- weighted Rota-Baxter Jacobi-Jordan algebra $(A_{\mathcal{I}},{\mathcal{I}})$.
	\end{remark}
In the following, we will prove another relevant result that will be useful in the next section to construct
the cohomology of $\lambda$-weighted Rota-Baxter Jacobi-Jordan algebras.
	\begin{theorem} $\label{Ta}$
		Let $(V,\mathcal{T})$ be a representation of a $\lambda$- weighted Rota-Baxter Jacobi-Jordan algebra $(A,\mathcal{I})$.   Define the map 
		\begin{eqnarray}
			\tilde{\rho}:A\rightarrow gl(V), \tilde{\rho}(x)u:= \rho(\mathcal{I}x)u-\mathcal{T}(\rho(x)u), \forall x\in A, u\in V. \nonumber		
		\end{eqnarray}
		Then, $\tilde{\rho}$ defines a representation of the Jacobi-Jordan algebra $A_{\mathcal{I}}$ on $V.$ Moreover,
		 $$(\tilde{V}=(V,\tilde{\rho}),\mathcal{T})$$ is a representation of the $\lambda$- weighted Rota-Baxter Jacobi-Jordan algebra $(A_{\mathcal{I}},\mathcal{I}).$	
	\end{theorem}
	\begin{proof} 
		Let $x,y\in A$ and $u\in V$, then, we have:
		\begin{eqnarray}			
			&&\tilde{\rho}(x)\tilde{\rho}(y)u+	\tilde{\rho}(y)\tilde{\rho}(x)u
			= \tilde{\rho}(x)\Big(\rho(\mathcal{I}y)u)-\mathcal{T}(\rho(y)u\Big)+\tilde{\rho}(y)\Big(\rho(\mathcal{I}x)u)-\mathcal{T}(\rho(x)u\Big)\nonumber\\
			&&=\rho(\mathcal{I}x)\Big(\rho(\mathcal{I}y)u-\mathcal{T}(\rho(y)u)\Big)-\mathcal{T}\Big(\rho(x)\rho(\mathcal{I}y)u-\rho(x)\mathcal{T}(\rho(y)u)\Big)
			+\rho(\mathcal{I}y)\Big(\rho(\mathcal{I}x)u\nonumber\\
			&&-\mathcal{T}(\rho(x)u)\Big)-\mathcal{T}\Big(\rho(y)\rho(\mathcal{I}x)u-\rho(y)\mathcal{T}(\rho(x)u)\Big)
			=\rho(\mathcal{I}x)\rho(\mathcal{I}y)u-\mathcal{T}\Big(\rho(\mathcal{I}x)\rho(y)u
			\nonumber\\
			&&+\cancel{\rho(x)\mathcal{T}(\rho(y)u)}+\lambda\rho(x)\rho(y)u\Big)
			-\mathcal{T}\Big(\rho(x)\rho(\mathcal{I}(y))u-\cancel{\rho(x)\mathcal{T}\rho(y)u}\Big) 
			+\rho(\mathcal{I}y)\rho(\mathcal{I}x)u
			\nonumber\\
			&& 
			-\mathcal{T}\Big(\rho(\mathcal{I}y)\rho(x)u+\cancel{\rho(y)\mathcal{T}(\rho(x)u)}+\lambda\rho(y)\rho(x)u\Big)
			-\mathcal{T}\Big(\rho(y)\rho(\mathcal{I}x)(u)-\cancel{\rho(y)\mathcal{T}\rho(x)u}\Big) \mbox{ ( by (\ref{rpRJJa}) )} \nonumber\\
			&&=\Big(\rho(\mathcal{I}x)\rho(\mathcal{I}y)u+\rho(\mathcal{I}y)\rho(\mathcal{I}x)u\Big) -\mathcal{T}\Big(\rho(\mathcal{I}x)\rho(y)u+\lambda\rho(x)\rho(y)u+\rho(x)\rho(\mathcal{I}y)u\Big) \nonumber\\
			&&-\mathcal{T}\Big(\rho(\mathcal{I}y)\rho(x)u+\lambda\rho(y)\rho(x)u+\rho(y)\rho(\mathcal{I}x)u\Big) 
			=-\rho(\mathcal{I}x\ast \mathcal{I}y)u+\mathcal{T}(\rho(x\ast_{\mathcal{I}} y)u) \mbox{( by (\ref{rpJJa}) and (\ref{TO}))}\nonumber\\
			&&=-\rho(\mathcal{I}(x\ast_{\mathcal{I}} y))u+\mathcal{T}(\rho(x\ast_{\mathcal{I}} y)u) 
			=-\tilde{\rho}(x\ast_{\mathcal{I}}y)u. \nonumber	
		\end{eqnarray}
		Therefore, $\tilde{\rho}$ is a representation of $A_{\mathcal{I}}.$ Next, we compute 
		\begin{eqnarray}			
		&&	\tilde{\rho}(\mathcal{I}x)\mathcal{T}(u)\nonumber\\
		&&=\rho(\mathcal{I}^{2}(x))(\mathcal{T}u)-\mathcal{T}(\rho(\mathcal{I}x)(\mathcal{T}u)) \nonumber \\
			&&=\mathcal{T}\Big(\rho(\mathcal{I}^2(x))u+\rho(\mathcal{I}x)(\mathcal{T}u)+\lambda\rho(\mathcal{I}x)u\Big)-\mathcal{T}^{2}\Big(\rho(\mathcal{I}x)u+\rho(x)(\mathcal{T}u)+\lambda\rho(x)u\Big) \mbox{ ( by (\ref{rpRJJa}) )}\nonumber\\
			&&=\mathcal{T}\Big(\rho(\mathcal{I}^{2}(x))u-\mathcal{T}(\rho(\mathcal{I}x)u)\Big)+\mathcal{T}\Big(\rho(\mathcal{I}x)(\mathcal{T}u)-\mathcal{T}(\rho(x)(\mathcal{T}u))\Big)+\mathcal{T}\Big(\lambda\rho(\mathcal{I}x)u)-\lambda\mathcal{T}(\rho(x)u)\Big)\nonumber\\
			&&=\mathcal{T}\Big(\tilde{\rho}(\mathcal{I}x)u+\tilde{\rho}(x)\mathcal{T}u+\lambda\tilde{\rho}(x)u\Big).\nonumber
		\end{eqnarray}	
		This ends of the proof.		
	\end{proof}
		\begin{proposition}
		Let $(A,\mathcal{I})$ be a $\lambda$-weighted  Rota-Baxter Jacobi-Jordan algebra  such that $\mathcal{I}$ is quasi-idempotent i.e $\mathcal{I}^{2}=-\mathcal{I}$. Then the following result holds:
		\begin{eqnarray}
			\mathcal{I}(x)\ast\mathcal{I}(y)=-\mathcal{I} (x\ast y), \forall x,y \in ImT.		
		\end{eqnarray}
		When $\mathcal{I}$ is idempotent i.e $\mathcal{I}^{2}=\mathcal{I}$ and $\lambda=-1$, then $\mathcal{I}$ is a morphism from $Im\mathcal{I}$ to $Im\mathcal{I}$.
	\end{proposition}
	\begin{proof}
	Suppose that $\mathcal{I}$ is quasi-idempotent	and let $x=\mathcal{I}(z), y=\mathcal{I}(t)$ with $z,t \in A$. Then we have
		\begin{eqnarray}
			&&\mathcal{I}(x)\ast \mathcal{I}(y)=\mathcal{I}(\mathcal{I}(z))\ast \mathcal{I}(\mathcal{I}(t))=
			\mathcal{I}( \mathcal{I}(z)\ast \mathcal{I}^2(t)+\mathcal{I}^2(z)\ast\mathcal{I}(t)+\lambda \mathcal{I}(z)\ast \mathcal{I}(t) ) \mbox{( by (\ref{ro}) )}\nonumber\\
			&&=\mathcal{I}( -\lambda\mathcal{I}(z)\ast \mathcal{I}(t) )=-\lambda\mathcal{I}(x\ast y).\nonumber
		\end{eqnarray}
		 When $\mathcal{I}^{2}=\mathcal{I}$ and $\lambda=-1$ we obtain automatically $\mathcal{I}(x)\ast \mathcal{I}(y)=\mathcal{I}(x\ast y), \forall x,y \in Im\mathcal{I}$.	
	\end{proof}	
	\begin{proposition}
		Let $(V,\rho)$ be a representation of a 0-weighted Rota-Baxter Jacobi-Jordan algebra $(A,\mathcal{I})$, $V^{\ast}$ the dual of $V$. Define  linear maps $\rho^{\ast}: A\longrightarrow gl(V^{\ast})$ and $\mathcal{T}^{\ast}: V^{\ast}\rightarrow V^{\ast}$ by: 
		\begin{eqnarray}
			(\rho^{\ast}(x)\theta)(u):=\theta\Big(\rho(x)u\Big) \mbox{ and }(\mathcal{T}^{\ast} \theta)(u):=-\theta(\mathcal{T}u), \forall x\in A, \forall u\in V, \forall \theta \in V^{\ast}.		
		\end{eqnarray}
	Then, $(V^{\ast},\rho^{\ast},\mathcal{T}^{\ast})$ is a representation of $(A,\mathcal{I})$.		
	\end{proposition}
	\begin{proof}
		It is proved that $(V^{\ast},\rho^{\ast})$ is a representation of the Jacobi-Jordan algebra $(A,\ast)$ \cite{oa}. Next, pick, $ x\in A,  u\in V$ and $\theta \in V^{\ast}$, then we have
		\begin{eqnarray}						
		&&	\Big[(\rho^{\ast}(\mathcal{I}x)(\mathcal{T}^{\ast}\theta)-\mathcal{T}^{\ast}\Big(\rho^{\ast}(\mathcal{I}x)\theta+\rho^{\ast}(x)(\mathcal{T}^{\ast}\theta)\Big)\Big](u)\nonumber\\
			&&=(\mathcal{T}^{\ast}\theta)\Big(\rho(\mathcal{I}x)u)\Big)+(\rho^{\ast}(\mathcal{I}x)\theta)(\mathcal{T}u)-(\rho^{\ast}(x)\mathcal{T}^{\ast}\theta)(\mathcal{T}u)\nonumber\\
			&&=-\theta\Big(\mathcal{T}(\rho(\mathcal{I}x)u)\Big)+\theta\Big(\rho(\mathcal{I}x)\mathcal{T}u\Big)-\theta\Big(\mathcal{T}(\rho(x)\mathcal{T}u)\Big)\nonumber\\
			&&=-\theta\Big(\mathcal{T}(\rho(\mathcal{I}x)u)\Big)+\theta\Big[\mathcal{T}\Big(\rho(\mathcal{I}x)u+\rho(x)(\mathcal{T}u)\Big)\Big]-\theta\Big(\mathcal{T}(\rho(x)\mathcal{T}u)\Big) \mbox{ ( by (\ref{rpRJJa}) )}\nonumber\\
			&&=0.\nonumber
		\end{eqnarray}
		This ends the proof.
	\end{proof}
	\section{Rota-Baxter paired operators}\label{s3}
	Let $(A,\ast)$ be a Jacobi-Jordan algebra and $(V, \rho)$ be a representation of it. In this section, weighted
	Rota-Baxter paired operators that are related to weighted Rota-Baxter Jacobi-Jordan algebras together with their representations, are introduced. After giving some results, a characterization of these operators is given.
	\begin{definition} $\label{D}$
		Let $(A,\ast)$ be a Jacobi-Jordan algebra, $(V, \rho)$ be a representation of it and $\lambda \in \mathbb{K}$. A pair $(\mathcal{I},\mathcal{T})$ of linear maps $\mathcal{I}: A \longrightarrow A$, $\mathcal{T}:V \longrightarrow V$ is called a $\lambda$-weighted Rota-Baxter paired operators if the following conditions  hold:
		\begin{eqnarray}
		&&	\mathcal{I}(x)\ast\mathcal{I}(y)=\mathcal{I}\Big(\mathcal{I}x\ast y+x\ast\mathcal{I}y+\lambda (x\ast y)\Big), \forall x,y \in A,\label{An} \\
		&&	\rho(\mathcal{I}x)\mathcal{T}(u)=\mathcal{T}\Big(\rho(\mathcal{I}x)(u)+\rho(x)\mathcal{T}(u)+\lambda\rho(x)(u)\Big), \forall x\in A, u\in V. \label{Am}				
		\end{eqnarray}
		\begin{remark} $\label{R}$
			Observe that the identity (\ref{An}) shows that the pair $(A,\mathcal{I})$ must be a $\lambda$-weighted Rota-Baxter Jacobi-Jordan algebra while the identity (\ref{Am}) is equivalent to say that $(V,\mathcal{T})$ is a representation of $(A,\mathcal{I}).$	
		\end{remark}	
	\end{definition}
	According to Remark \ref{R}, we get the following result:
	\begin{proposition} \label{P}
		Let $(V,\rho)$ be a representation of a  Jacobi-Jordan algebra $(A,\ast)$, $\lambda \in \mathbb{K}$ and $\mathcal{I}:A \longrightarrow A$, $\mathcal{T}:V \longrightarrow V$ be linear maps. Then the pair $(\mathcal{I},\mathcal{T})$ is a $\lambda$-weighted Rota-Baxter paired operators if and only if \begin{enumerate}
			\item The pair $(A, \mathcal{I})$ is a $\lambda$-weighted Rota-Baxter Jacobi-Jordan algebra.
			\item The pair $(V,\mathcal{T})$ is a representation of $(A, \mathcal{I}).$	
		\end{enumerate}		
	\end{proposition}
	\begin{proof}
		It follows by Definition \ref{D}.	
	\end{proof}
\begin{proposition}\label{ds}
Let $(A,\ast)$ be a Jacobi-Jordan algebra. Then, the  direct sum $A\oplus A$  carries a Jacobi-Jordan algebra structure with the map $\ast_{\lambda}$ given by
\begin{eqnarray}
	(x,x')\ast_{\lambda}(y,y')= (x\ast y,x\ast y'+x\ast 
	y'+\lambda (x'\ast y')),\ \forall x,x',y,y'\in A.
\end{eqnarray}
We denote this algebra by $(A\oplus A)_{\lambda}.$ Furthermore, if $ (V, \rho)$ is a representation of the Jacobi-Jordan algebra $(A,\ast)$, then $V\oplus V$ can be endowed with a representation of the Jacobi-Jordan algebra $(A\oplus A)_{\lambda}$ with the action $\rho_{\lambda}$ given by 
\begin{eqnarray}
	\rho_{\lambda}(x,y)(u,v):= (\rho(x)u,\rho(x)v+\rho(y)u+\lambda\rho(y)v), \forall x,y \in A, u,v \in V.
\end{eqnarray}
\end{proposition}
\begin{proof}
First, let $(x,x'), (y,y'), (z,z') \in A^{\otimes 2}$ then we have:
	\begin{eqnarray}			
	&&\Big((x,x')\ast_{\lambda}(y,y')\Big)\ast_{\lambda}(z,z')+\Big((y,y')\ast_{\lambda}(z,z')\Big)\ast_{\lambda}(x,x')+\Big((z,z')\ast_{\lambda}(x,x')\Big)\ast_{\lambda}(y,y')\nonumber\\
	&&= \Big((x\ast y)\ast z+(y\ast z)\ast x+(z\ast x)\ast y,\Big((x\ast y)\ast z'+(y\ast z')\ast x+(x\ast z')\ast y\Big) \nonumber\\	
		&&+\Big((x\ast y')\ast z+(y'\ast z)\ast x+(z\ast x)\ast y'\Big)
		+\Big((x'\ast y)\ast z+(y\ast z)\ast x'+(z\ast x')\ast y\Big) \nonumber \\
	&&	+\lambda\Big((x'\ast y')\ast z+(y'\ast z)\ast x'+(z\ast x')\ast y'\Big)
		+\lambda\Big((x\ast y')\ast z'+(y'\ast z')\ast x+(z'\ast x)\ast y'\Big) \nonumber \\
	&&	+\lambda\Big((x'\ast z')\ast y+(y\ast z')\ast x'+(x'\ast y)\ast z'\Big)
		+\lambda^{2}\Big((x'\ast y')\ast z'+(y'\ast z')\ast x'+(z'\ast x')\ast y'\Big)\Big) \nonumber\\
		&&=0 \mbox{(by the Jacobi-identity in $(A,\ast)$ ).} \nonumber
\end{eqnarray}
	Next, let $(x,y),(z,t) \in A^{\otimes 2}$ and $(u,v)\in V^{\times 2}$, then we obtain by straightforward computations,
	\begin{eqnarray}							
&&\rho_{\lambda}((x,y)\ast_{\lambda}(z,t))(u,v)+	\rho_{\lambda}((x,y))\rho_{\lambda}((z,t))(u,v)+\rho_{\lambda}((z,t))\rho_{\lambda}((x,y))(u,v)\nonumber\\
&&=
\Big(\rho(x\ast z)u+\rho(x)\rho(z)u+\rho(z)\rho(x)u,\  
\rho(x\ast z)v+\rho(x)\rho(z)v+\rho(z)\rho(x)v \nonumber\\
&&+\rho(x\ast t)u+\rho(x)\rho(t)u+\rho(t)\rho(x)u
+\rho(y\ast z)u+\rho(y)\rho(z)u+\rho(z)\rho(y)u\nonumber\\
&&+\lambda\Big(\rho(y\ast t)u+\rho(t)\rho(y)u+\rho(y)\rho(t)u\Big)
+\lambda\Big(\rho(y\ast z)v+\rho(y)\rho(z)v+\rho(z)\rho(y)v\Big)\nonumber\\
&&+\lambda\Big(\rho(y\ast t)u+\rho(y)\rho(t)u+\rho(t)\rho(y)u\Big)
+\lambda^{2}\Big(\rho(y\ast t)v+\rho(y)\rho(t)v+\rho(t)\rho(y)v\Big)\Big)\nonumber\\
&&=(0,0) \mbox{( by (\ref{rpJJa}) ).}\nonumber	
\end{eqnarray}
Hence, the conclusion follows.
\end{proof}
Let $(A,\ast)$ be a Jacobi-Jordan algebra and $ (V, \rho)$ a representation on it. Consider now the Jacobi-Jordan algebra $(A\oplus A)_{\lambda}$ and its represnetation $(V\oplus V,\rho_{\lambda})$ given in Proposition \ref{ds}. Then, we obtain
\begin{proposition}
The direct sum $A\oplus A\oplus V\oplus V$ carries the semidirect product Jacobi-Jordan algebra structure with the product $\ltimes$ explicitly given by:
\begin{eqnarray}
	\Big(x,x',u,u'\Big)\ltimes \Big(y,y',v,v'\Big):=\Big(x\ast y,x\ast y'+x'\ast y+\lambda (x'\ast y'),\rho(x)v+\rho(y)u,\rho(x)v'+\rho(x')v\nonumber \\
	+\lambda\rho(x')v'+\rho(y)u'+\rho(y')u+\lambda\rho(y')u'\Big), \forall x,x',y,y' \in A, \forall u,u',v,v' \in V. \nonumber
\end{eqnarray}
\end{proposition}
\begin{proof}
Let $ \Big(x,x',u,u'\Big),\Big(y,y',v,v'\Big),\Big(z,z',w,w'\Big) \in A\times A\times V\times V$, then by straightforward computations, we get
\begin{eqnarray}
&&\Big((x,x',u,u')\ltimes(y,y',v,v')\Big)\ltimes(z,z',w,w')+\Big((y,y',v,v')\ltimes(z,z',w,w')\big)\ltimes (x,x',u,u') \nonumber \\
&&+\Big((z,z',w,w')\ltimes(x,x',u,u')\Big)\ltimes(y,y',v,v')=\Big((x\ast y)\ast z+(y\ast z)\ast x+(z\ast x)\ast y,\nonumber \\
&&\Big((x\ast y)\ast z'+(y\ast z')\ast x+(z'\ast x)\ast y\Big)+\Big((x\ast y')\ast z+(y'\ast z)\ast x+(z\ast x)\ast y'\Big)\nonumber\\
&&+\Big((x'\ast y)\ast z+(y\ast z)\ast x'+(z\ast x')\ast y\Big)+\lambda^{2}\Big((x'\ast y')\ast z'+(y'\ast z')\ast x'+(z'\ast x')\ast y'\Big)\nonumber \\
&&+\lambda\Big((x'\ast y')\ast z+(y'\ast z)\ast x'+(z\ast x')\ast y'\Big)+\lambda\Big((x\ast y')\ast z'+(y'\ast z')\ast x+(z'\ast x)\ast y'\Big) \nonumber \\
&&+\lambda\Big((x'\ast y)\ast z'+(y\ast z')\ast x'+(z'\ast x')\ast y\Big),\big(\rho(x\ast y)w+\rho(x)\rho(y)w+\rho(y)\rho(x)w\Big)\nonumber 
\end{eqnarray}
\begin{eqnarray}
&&+\Big(\rho(y\ast z)u+\rho(y)\rho(z)u+\rho(z)\rho(y)u\Big)+\Big(\rho(z\ast x)v+\rho(z)\rho(x)v+\rho(x)\rho(z)v\Big), \nonumber \\
&&\Big(\rho(y'\ast z)u+\rho(y')\rho(z)u+\rho(z)\rho(y')u\Big)+\Big(\rho(y\ast z')u+\rho(y)\rho(z')u+\rho(z')\rho(y)u\Big)\nonumber \\
&&+\lambda\Big(\rho(y'\ast z')u+\rho(y')\rho(z')u+\rho(z')\rho(y')u\Big)+\Big(\rho(x\ast z')v+\rho(x')\rho(z)v+\rho(z)\rho(x')v\Big)\nonumber \\
&&+\Big(\rho(x'\ast z)v+\rho(x')\rho(z)v+\rho(z)\rho(x')v\Big)+\lambda\Big(\rho(z'\ast x')v+\rho(z')\rho(x')v+\rho(x')\rho(z')v\Big)\nonumber\\
&&+\Big(\rho(x\ast y')w+\rho(x)\rho(y')w+\rho(y')\rho(x)w\Big)+\Big(\rho(x'\ast y)w+\rho(x')\rho(y)w+\rho(y)\rho(x')w\Big)\nonumber\\
&&+\lambda\Big(\rho(x'\ast y')w+\rho(x')\rho(y')w+\rho(y')\rho(x')w\Big)+\Big(\rho(y\ast z)u'+\rho(y)\rho(z)u'+\rho(z)\rho(y)u'\Big)\nonumber\\
&&+\lambda\Big(\rho(y\ast z')u'+\rho(y)\rho(z')u'+\rho(z')\rho(y)u'\Big)+\lambda\Big(\rho(y'\ast z)u'+\rho(y')\rho(z)u'+\rho(z)\rho(y')u'\Big)\nonumber\\
&&+\lambda^{2}\Big(\rho(y'\ast z')u'+\rho(y')\rho(z')u'+\rho(z')\rho(y')u'\Big)+\Big(\rho(z\ast x)v'+\rho(z)\rho(x)v'+\rho(x)\rho(z)v'\Big) \nonumber \\
&&+\lambda\Big(\rho(z\ast x')v'+\rho(z)\rho(x')v'+\rho(x')\rho(z)v'\Big)+\Big(\rho(z'\ast x)v'+\rho(z')\rho(x)v'+\rho(x)\rho(z')v'\Big) \nonumber\\
&&+\lambda^{2}\Big(\rho(z'\ast x')v'+\rho(z')\rho(x')v'+\rho(x')\rho(z')v'\Big)
+\Big(\rho(x\ast y)w'+\rho(x)\rho(y)w'+\rho(y)\rho(x)w'\Big)\nonumber\\
&&+\lambda\Big(\rho(x\ast y')w'+\rho(x)\rho(y')w'+\rho(y')\rho(x)w'\Big)+\lambda\Big(\rho(x'\ast y)w'+\rho(x')\rho(y)w'+\rho(y)\rho(x')w'\Big) \nonumber \\
&&+\lambda^{2}\Big(\rho(x'\ast y')w'+\rho(x')\rho(y')w'+\rho(y')\rho(x')w'\Big)\Big). \nonumber		
\end{eqnarray}
The conclusion follows.
\end{proof}
We denote this Jacobi-Jordan algebra by $(A\oplus A\oplus V\oplus V)_{\lambda}$. We get the following characterization of a $\lambda$-weighted Rota-Baxter paired operators.
\begin{proposition}
	Let $(A,\ast)$ be a Jacobi-Jordan algebra and $(V,\rho)$ be a representation of it. A pair $(\mathcal{I},\mathcal{T})$ of linear maps
	$\mathcal{I}:A \rightarrow A$ and $ \mathcal{T}:V \rightarrow V$ forms a $\lambda$-weighted Rota-Baxter paired operators if and only if the set
	$$\mathcal{G}r_{(\mathcal{I},\mathcal{T})}:=\{(\mathcal{I}x, x, \mathcal{T}u, u): x \in g, u \in V\}\subset A\oplus A\oplus V\oplus V$$
	is a Jacobi-Jordan subalgebra of $(A\oplus A\oplus V\oplus V)_{\lambda}$.
\end{proposition}
\begin{proof}
	Let $x,y\in A$ and $u,v\in V.$ Then, we have
	\begin{eqnarray}
&&	\Big(\mathcal{I}(x),x,\mathcal{T}(u),u\Big)\ltimes \Big(\mathcal{I}(y),y,\mathcal{T}(v),v\Big)\nonumber\\
=&&\Big(\mathcal{I}(x)\ast \mathcal{I}(y), \mathcal{I}(x)\ast y+x\ast \mathcal{I}(y)+\lambda(x\ast y),\rho(\mathcal{I}x)\mathcal{T}v+\rho(\mathcal{I}y)\mathcal{T}u, \nonumber\\&&
\rho(\mathcal{I}x)v+\rho(x)\mathcal{T}v+\lambda\rho(x)v+\rho(\mathcal{I}y)u+\rho(y)\mathcal{T}u+\lambda\rho(y)u\Big)\nonumber
	\end{eqnarray}
Hence, $\Big(\mathcal{I}(x),x,\mathcal{T}(u),u\Big)\ltimes \Big(\mathcal{I}(y),y,\mathcal{T}(v),v\Big)\in \mathcal{G}r_{(\mathcal{I},\mathcal{T})} \Longleftrightarrow \Big(\mbox{ (\ref{An}) and (\ref{Am}) hold} \Big).$
\end{proof}
\section{Zigzag cohomology of $\lambda$- weighted Rota-Baxter Jacobi-Jordan algebras}\label{s4}
	
	In this section, we first recall the zigzag cohomology of a Jacobi-Jordan algebra with coefficients in a
	representation \cite{Momo5}. Then we define the zigzag cohomology of a weighted Rota-Baxter Jacobi-Jordan algebra with coefficients in a representation. This zigzag cohomology is obtained as a byproduct of the zigzag cohomology of
	the underlying Jacobi-Jordan algebra with the one of the weighted Rota-Baxter operators.
	
	Let $(A,\ast)$ be a Jacobi-Jordan algebra and $(V, \rho)$ be a representation of it. The zigzag cohomology of $(A,\ast)$ with coefficients in $V$ \cite{Momo5} is given by the cohomology of the cochain complex 
	$(\mathcal{C}^{\bullet}(A,V),\mathcal{A}^{\bullet}(A,V),d^{\bullet},\delta^{\bullet})$, where for all  integer $n\geq 0$, 
	\begin{eqnarray}
		&&d^n: \mathcal{C}^{n}(A,V) \longrightarrow \mathcal{C}^{n+1}(A,V), \nonumber \\
		&&d^n f(x_1, \cdots, x_{n+1})
		=\sum\limits_{i=1}^{n+1}\rho(x_i)f(x_1,\cdots,\widehat{x_i},\cdots, x_{n+1})\nonumber\\
		&&+\sum\limits_{1\leq i< j\leq n+1} f(x_i\ast x_j, x_1, \cdots, \widehat{x}, \cdots, \widehat{x_j}, \cdots, x_{n+1}), \nonumber
	\end{eqnarray}
	\begin{eqnarray}
		&&\delta^n: \mathcal{A}^{n}(A,V) \longrightarrow \mathcal{C}^{n+1}(A,V), \nonumber \\
		&&\delta^n g(x_1, \cdots, x_{n+1})
		=\sum\limits_{i=1}^{n+1}\rho(x_i)g(x_1,\cdots,\widehat{x_i},\cdots, x_{n+1})\nonumber\\
		&&-\sum\limits_{1\leq i< j\leq n+1} g(x_i\ast x_j, x_1, \cdots, \widehat{x}, \cdots,\widehat{x_j},\cdots,x_{n+1})  \nonumber	
	\end{eqnarray}
for all  $(x_{1},x_{2},\cdot\cdot\cdot,x_{n+1}) \in A^{\otimes (n+1)}$, $ f \in \mathcal{C}^{n}(A,V)$ and $g\in\mathcal{A}^{n}(A,V).$\\
\\
	Let $(A,\mathcal{I})$ be a $\lambda$-weighted Rota-Baxter Jacobi-Jordan algebra and $(V,\mathcal{T})$ be a representation of it. We have proved in Theorem $\ref{Ta}$ that the pair $\tilde{V}=(V,\tilde{\rho})$ is a representation of the Jacobi-Jordan  algebra $A_{\mathcal{I}}$. Therefore, as a similar way, we can define the corresponding zigzag cohomology. Here, $(\mathcal{C}^{\bullet}(A_{\mathcal{I}},\tilde{V}),\mathcal{A}^{\bullet}(A_{\mathcal{I}},\tilde{V}), \tilde{d}^{\bullet},\tilde{\delta}^{\bullet})$ is the cochain complex with 
	\begin{eqnarray}		
	&&	\tilde{d}^n: \mathcal{C}^{n}(A_{\mathcal{I}},\tilde{V}) \longrightarrow \mathcal{C}^{n+1}(A_{\mathcal{I}},\tilde{V}), \nonumber \\
		&&\tilde{d}^n f(x_1, \cdots, x_{n+1})=\sum\limits_{i=1}^{n+1}\tilde{\rho}(x_i)f(x_1,\cdots,\widehat{x_i},\cdots, x_{n+1})\nonumber\\
		&&+\sum\limits_{1\leq i< j\leq n+1} f(x_i\ast_{\mathcal{I}} x_j, x_1, \cdots, \widehat{x_{i}}, \cdots, \widehat{x_j}, \cdots, x_{n+1}), \nonumber
	\end{eqnarray}
	\begin{eqnarray}	
	&&\tilde{\delta}^n: \mathcal{A}^{n}(A_{\mathcal{I}},\tilde{V}) \longrightarrow \mathcal{C}^{n+1}(A_{\mathcal{I}},\tilde{V}), \nonumber \\
	&&	\tilde{\delta}^n g(x_1, \cdots, x_{n+1})=\sum\limits_{i=1}^{n+1}\tilde{\rho}(x_i)g(x_1,\cdots,\widehat{x_i},\cdots, x_{n+1})\nonumber\\
		&&-\sum\limits_{1\leq i< j\leq n+1} g(x_i\ast_{\mathcal{I}} x_j, x_1, \cdots, \widehat{x_{i}}, \cdots,\widehat{x_j},\cdots,x_{n+1})
		\nonumber
	\end{eqnarray}
	for all  $(x_{1},x_{2},\cdot\cdot\cdot,x_{n+1}) \in A^{\otimes(n+1)},$ $f \in \mathcal{C}^{n}(A_{\mathcal{I}},\tilde{V})$ and $g\in\mathcal{A}^{n}(A_{\mathcal{I}},\tilde{V}).$
	The corresponding  cohomology spaces are called the zigzag cohomology of $\mathcal{I}$ with coefficients in the representation $\mathcal{T}.$ \\
	\\
	From this moment on, let us define the zigzag cohomology of the $\lambda$-weighted Rota-Baxter Jacobi-Jordan algebra $(A,\mathcal{I})$ with coefficients in the representation $(V,\mathcal{T}).$
	Firstly, we consider the two cochain complexes, namely
	\begin{itemize}
		\item [$\diamond$] $(\mathcal{C}^{\bullet}(A,V), \mathcal{A}^{\bullet}(A,V), d^{\bullet},\delta^{\bullet})$ defining the zigzag cohomology of the Jacobi-Jordan algebra
		$(A,\ast)$ with coefficients in the representation $(V,\rho)$,
		\item [$\diamond$] $(\mathcal{C}^{\bullet}(A_{\mathcal{I}},\tilde{V}),\mathcal{A}^{\bullet}(A_{\mathcal{I}},\tilde{V}), \tilde{d}^{\bullet},\tilde{\delta}^{\bullet})$ defining the zigzag cohomology of $\mathcal{I}$ with coefficients in the representation $\mathcal{T}.$	
	\end{itemize}

\begin{proposition}
	Let $(A, \mathcal{I})$ be a $\lambda$- weighted Rota-Baxter Jacobi-Jordan algebra and $(V, \mathcal{T})$ be a representation of it. Then, 
	the maps $\phi^{n}_{1}:\mathcal{C}^{n}(A,V)\rightarrow \mathcal{C}^{n}(A_{\mathcal{I}},V)$ and  $\phi^{n}_{2}: \mathcal{A}^{n}(A,V)\rightarrow \mathcal{C}^{n}(A_{\mathcal{I}},V)$ defined by: 
	\begin{equation}		
		\phi^{n}_{1}: \left\{\begin{array}{ll}
			\phi^{0}_{1}(v):=v & \text{ if } n=0 ,\forall v\in V\\
			\phi^{1}_{1}(f):=f\circ \mathcal{I}-\mathcal{T}\circ f& \text{ if } n=1, \forall f\in \mathcal{C}^{1}(A,V) \\
		\end{array} \right.
	\end{equation}
	
	\begin{equation}		
		\phi^{n}_{2}: \left\{\begin{array}{ll}
			\phi^{0}_{2}(v):=v & \text{ if } n=0, \forall v\in V \\
			\phi^{1}_{2}(f):=f\circ \mathcal{I}+\mathcal{T}\circ f & \text{ if } n=1, \forall f\in \mathcal{A}^{1}(A,V) \\
		\end{array} \right.
	\end{equation}
	satisfy
	\begin{eqnarray}
	\tilde{d}^0\circ\phi^{0}_{2}=\phi^{1}_{1}\circ\delta^0. \label{Zo}
	\end{eqnarray}
\end{proposition}
	\begin{proof}
	Let $x\in A$ and $v\in V$, then we have 
			\begin{eqnarray}
				(\tilde{d}^0\circ\phi^{0}_{2}(v))(x)=\tilde{d}^0(v)(x)
				=\tilde{\rho}(x)v
				=\rho(\mathcal{I}x)v-\mathcal{I}(\rho(x)v),\label{Pa}
			\end{eqnarray}
			\begin{eqnarray}
				(\phi^{1}_{1}\circ\delta^0(v))(x)
				=\delta^0(v)(\mathcal{I}x)-\mathcal{T}(\delta^0(v)x)\
				=\rho(\mathcal{I}x)v-\mathcal{T}(\rho(x)v).\label{Va}	
			\end{eqnarray}
			 The conclusion follows from (\ref{Pa}) and (\ref{Va}).
		 	\end{proof}
		Let $(A, \mathcal{I})$ be a $\lambda$- weighted Rota-Baxter Jacobi-Jordan algebra and $(V, \mathcal{T})$ be a representation of it. 	
	For each $n \in \{0, 1\}$, we define spaces $\mathcal{C}^{n}_{RB}(A,V)$ and $\mathcal{A}^{n}_{RB}(A,V)$ by
	\begin{equation}
		\mathcal{C}^{n}_{RB}(A,V) := \left\{\begin{array}{ll}
			\mathcal{C}^{0}(A,V)=V& \text{ if } n=0, \\
			\mathcal{C}^{n}(A,V)\oplus \mathcal{C}^{n-1}(A_{\mathcal{I}},\tilde{V})  & \text{ if } n=1, \\
		\end{array} \right.
	\end{equation}
\begin{equation}
	\mathcal{A}^{n}_{RB}(A,V) := \left\{\begin{array}{ll}
		\mathcal{A}^{0}(A,V)=V& \text{ if } n=0, \\
		\mathcal{A}^{n}(A,V)\oplus \mathcal{A}^{n-1}(A_{\mathcal{I}},\tilde{V})  & \text{ if } n=1. \\
	\end{array} \right.
\end{equation}		
We also define the maps 
\begin{eqnarray}
	d_{RB}^n:\mathcal{C}^{n}_{RB}(A,V)\rightarrow \mathcal{C}^{n+1}_{RB}(A,V) \nonumber 	
\end{eqnarray}
\begin{equation}		
d_{RB}^n: \left\{\begin{array}{ll}
	d_{RB}^0(v,v):=(d^0(v),-v)& \text{ if } v\in \mathcal{C}^{0}_{RB}(A,V), \\
	d_{RB}^n(f,g):=(d^n(f),-\tilde{d}^{n-1}(g)-\phi^{n}_{1}(f))  & \text{ if } (f,g)\in \mathcal{C}^{n}_{RB}(A,V), \\
\end{array} \right.
\end{equation}

\begin{eqnarray}
	\delta_{RB}^n:\mathcal{A}^{n}_{RB}(A,V)\rightarrow \mathcal{C}^{n+1}_{RB}(A,V) \nonumber 	
\end{eqnarray}
\begin{equation}		
	\delta_{RB}^n(v,v): \left\{\begin{array}{ll}
		\delta_{RB}^0(v,v):=(\delta^0(v),-v)& \text{ if } v\in \mathcal{A}^{0}_{RB}(A,V), \\
		\delta_{RB}^n(f,g):=(\delta^n(f),-\tilde{\delta}^{n-1}(g)-\phi^{n}_{2}(f))  & \text{ if } (f,g)\in \mathcal{A}^{n}_{RB}(A,V), \\
	\end{array} \right.
\end{equation} 
\begin{proposition}
	With previous notations, we have:
\begin{eqnarray}
	d_{RB}^1\circ\delta_{RB}^0=0. \label{Yo}
\end{eqnarray}
\end{proposition}
\begin{proof}
Let $x\in A$ and $u,v\in V,$ then we have: 
\begin{eqnarray}
	(d_{RB}^1\circ\delta_{RB}^0)(u,v)&=&d_{RB}^1\Big(\delta(u),-\tilde{\delta}^0(v)-\phi^{0}_{2}(u)\Big)\nonumber\\
	&=&d_{RB}\Big(\delta^0(u),-\tilde{\delta}^0(v)-u\Big)\nonumber\\
	&=&	\Big(d^1\circ\delta^0(u),\tilde{d}^1\circ\tilde{\delta}^0(v)+\tilde{d}^0(u)-\phi_{1}^{1}(\delta^0(u)) \Big) \nonumber \\
		&=&	\Big(d^1\circ\delta^0(u),\tilde{d}^1\circ\tilde{\delta}^0(v)+\tilde{d}^0(u)-\tilde{d}^0(u)) \Big) \mbox{( by (\ref{Zo}) )}\nonumber \\
		&=&(0,0) \mbox{ (since  $d\circ\delta=\tilde{d}\circ\tilde{\delta}=0$)}.\nonumber
\end{eqnarray}
Hence, we obtain (\ref{Yo}) and the proof follows.
\end{proof}
 	This shows that $(\mathcal{C}^{\bullet}_{RB}(A,V),\mathcal{A}^{\bullet}_{RB}(A,V), d_{RB}^{\bullet},\delta_{RB}^{\bullet})$ is a cochain complex.
	Let $\mathcal{Z}_{RB}^{\bullet}(A,V)$ and $\mathcal{B}_{RB}^{\bullet}(A,V)$ denote the cocycle and coboundary spaces, respectively. Then we have $\mathcal{B}_{RB}^n(A,V)\subset \mathcal{Z}^n_{RB}(A,V)$ and the corresponding quotient is 
	\begin{eqnarray}
	\mathcal{H}_{RB}^n(A,V):=\mathcal{Z}_{RB}^n(A,V)/\mathcal{B}_{RB}^n(A,V), \mbox{$n=0,1.$}	
	\end{eqnarray}
To compute some cohomology spaces of $\lambda$-weighted Jacobi-Jordan algebras with coefficients in representations, let consider the following
\begin{definition} $\label{De}$
	Let $(A, \mathcal{I})$ be a $\lambda$-weighted Rota-Baxter Jacobi-Jordan algebra and $(V,\mathcal{T})$ be a representation of it. A pair $(\eta,v) \in \mathcal{H}om(A,V)\oplus V$ is said to be an antiderivation of $(A, \mathcal{I})$ with values in the representation $(V,\mathcal{T})$ if they satisfy
	\begin{eqnarray}
	&&	\eta(x\ast y)=-\rho(x)(\eta(y))-\rho(y)(\eta(x)), \forall x,y\in A,\label{ant1}\\
	&&	\eta(\mathcal{I}x)-\mathcal{T}(\eta(x))=\mathcal{T}(\rho(x)v)-\rho(\mathcal{I}x)(v), \forall x\in A. \label{ant2}		
	\end{eqnarray}	
\end{definition}
Observe that the first condition (\ref{ant1}) means that $\eta$ is an antiderivation of the Jacobi-Jordan  algebra $(A,\ast)$ with values in the representation $(V,\rho)$.
The second condition (\ref{ant2}) says that the obstruction of vanishing $\eta\circ\mathcal{I}-\mathcal{T}\circ\eta$ is measured by the presence of $v$.

Denote by $\mathcal{AD}er(A,V)$  the set of all antiderivations of the $\lambda$-weighted Rota-Baxter Jacobi-Jordan algebra $(A,\mathcal{I})$ with values in the representation $(V,\mathcal{T})$. 
	\begin{proposition}
		Let $(A, \mathcal{I})$ be a $\lambda$-weighted Rota-Baxter Jacobi-Jordan algebra and $(V,\mathcal{T})$ be a representation of it. Then, for all $v\in V$, $(D_v,-v)$  where
		\begin{eqnarray}
			D_v: A\rightarrow V, x\mapsto D_v(x):=\rho(x)v
		\end{eqnarray}
		is an antiderivation of $(A, \mathcal{I})$ with values in $(V,\mathcal{T})$, called an inner antiderivation of $(A, \mathcal{I})$ with values in $(V,\mathcal{T})$.
	\end{proposition}
	\begin{proof}
Let $x,y\in A$, then using (\ref{rpJJa}), we have first
\begin{eqnarray}
D_v(x\ast y)=\rho(x\ast y)v=-\rho(x)(\rho(y)v)-\rho(y)(\rho(x)v)
=-\rho(x)(D_v(y))-\rho(y)(D_v(x)).\nonumber
\end{eqnarray}
Next, we compute
\begin{eqnarray}
	D_v(\mathcal{I}x)-\mathcal{T}(D_v(x))=\rho(\mathcal{I}x)v-\mathcal{T}(\rho(x)v)
	=-\rho(\mathcal{I}x)(-v)+\mathcal{T}(\rho(x)(-v)).\nonumber
\end{eqnarray}
This ends the proof.
	\end{proof}	
Denote by  $Inn\mathcal{AD}er(A,V)$ the set of all  inner antiderivations of the $\lambda$-weighted Rota-Baxter Jacobi-Jordan algebra $(A,\mathcal{I})$ with values in the representation $(V,\mathcal{T})$. Then
\begin{eqnarray}
Inn\mathcal{AD}er(A,V):=\{(D_v,-v),v\in V\}=\{(\rho(-)v,-v),v\in V\}=\{\delta_{RB}^0(v,v),v\in V\}. \label{iad}
\end{eqnarray}
\begin{remark}
The notion of antiderivations (respectively inner antiderivations) of $(A,\mathcal{I})$ with values in the adjoint representation $(A,\mathcal{I})$ is called antiderivations (respectively inner antiderivations) of $(A,\mathcal{I})$. Let us denote their sets by $\mathcal{AD}er(A)$ and $Inn\mathcal{AD}er(A)$ respectively.
\end{remark}

\begin{example}\label{cant}
	\begin{enumerate}		
		\item[(i)] Let consider the adjoint representation of a $\lambda$-weighted Rota-Baxter Jacobi-Jordan algebra $(A,\mathcal{I})$, then for all $x\in A$, the pair $(L_{x},-x)$ where $ L_{x}(y):=x\ast y $ $\forall ,y\in A$ is an (inner) antiderivation of $(A,\mathcal{I})$.
		\item[(ii)] Let's consider the $0$-weighted Rota-Baxter Jacobi-Jordan algebra $(A,\mathcal{I})$ defined in Example \ref{F}.
	 where 
		\begin{eqnarray}
			\mathcal{I}=
			\begin{pmatrix}
				0&0	\\
				b&d	
			\end{pmatrix}, 
			b,d \in\mathbb{K}, d\neq 0.
		\end{eqnarray}
	Let $\delta: A\rightarrow A$ be a linear map such that the matrix of $\delta$ on the basis $\{ e_1,e_2\}$ is given by
	$\begin{pmatrix}
		a_{11}&a_{12}	\\
		a_{21}&a_{22}	
	\end{pmatrix} $  and $x=x_1e_1+x_2e_2\in A$ where $a_{11}, a_{21}, a_{12}, a_{22}, x_1, x_2$ are scalars.\\
	 Then, $(\delta,x)$ is an antiderivation of $(A,\mathcal{I})$ if and only if
	 \begin{eqnarray}
	 	&&\delta(e_i\ast e_j)=-e_i\ast \delta(e_j)-\delta(e_i)\ast e_j, \mbox{$ i,j=1,2,$} \label{a1}\\
	 	&& \delta(\mathcal{I}(e_k))-\mathcal{I}(\delta(e_k))=\mathcal{I}(e_k\ast x)-\mathcal{I}(e_k)\ast x, \mbox{$k=1,2$.}\label{a2}
	 \end{eqnarray}
Equations (\ref{a1}) and (\ref{a2}) hold if and only if
$\delta=\begin{pmatrix}
	a_{11}&0	\\
	a_{21}&-2a_{11}	
\end{pmatrix} $
and $x=(-a_{21}+\frac{3a_{11}b}{d})e_1+x_2e_2$.
Let the vector $v$ be represented by its column matrix. Then, we have
$$\mathcal{AD}er(A)=\left\{\left(\begin{pmatrix}
	a_{11}&0	\\
	a_{21}&-2a_{11}	
\end{pmatrix}, \begin{pmatrix}
-a_{21}+\frac{3a_{11}b}{d}	\\
x_2	
\end{pmatrix} \right), a_{11}, a_{21}, x_2\in\mathbb{R}\right\} \mbox{ and its dimension if $5$}.$$
Similarly for any $x\in A$, $(L_x,-x)$ where $L_x(y)=x\ast y,\ \forall y\in A$, is an (inner) antiderivation  of $(A,\mathcal{I})$ if and only if $x=x_1e_1+x_2e_2$ and the matrix of $L_x$ on the basis $\{ e_1,e_2\}$ is given by
$\begin{pmatrix}
	0&0	\\
	x_1&0	
\end{pmatrix}.$ Hence, the dimension of $Inn\mathcal{AD}er(A,V)$ is $2$ and
$$Inn\mathcal{AD}er(A)=\left\{\left(\begin{pmatrix}
	0&0	\\
	x_1&0	
\end{pmatrix},\begin{pmatrix}
-x_1\\
-x_2	
\end{pmatrix}\right), x_1, x_2\in\mathbb{R}\right\}.$$
	\end{enumerate}	
\end{example}	
\section*{Computations of $\mathcal{H}_{RB}^{0}(A,V)$ and $\mathcal{H}_{RB}^{1}(A,V)$}
An algebraic motivation is the interpretation of the low degree cohomology spaces in terms of
algebraic properties of the weighted Rota-Baxter Jacobi-Jordan algebra, which makes the cohomology spaces interesting invariant to compute.\\

Let $(A,\mathcal{I})$ be a $\lambda$-weighted Rota-Baxter Jacobi-Jordan algebra and $(V,\mathcal{T})$ be a representation of it. Let compute $\mathcal{H}_{RB}^{0}(A,V)$ and $\mathcal{H}_{RB}^{1}(A,V)$.
\begin{itemize}
\item $\mathcal{H}_{RB}^{0}(A,V)$\\
 An element $v\in V$ is in $\mathcal{Z}^{0}_{RB}(A,V)$ if and only if $d_{RB}(v,v)=(0,0)$ i.e $(d(v),-v)=(0,0)$. This holds only when $v=0$. Hence, $\mathcal{H}_{RB}^{0}(A,V)=\{0\}$.
 \item $\mathcal{H}_{RB}^{1}(A,V)$\\
 Let $(\eta,v)$ be an element of $(\mathcal{C}^{1}_{RB}(A,V)$. Then $(\eta,v)$ is in $\mathcal{Z}^{1}_{RB}(A,V)$ if and only if $d_{RB}(\eta,v)=(0,0)$ i.e $(d(\eta),-\tilde{d}(v)-\phi^{1}_{1}(\eta))=(0,0)$. So, we get the following identities:
 \begin{eqnarray}
 	&&\eta(x\ast y)=-\rho(x)(\eta(y))-\rho(y)(\eta(x)), \forall x, y \in A, \\
 &&	\eta(\mathcal{I}(x))-\mathcal{T}(\eta(x))=\mathcal{T}(\rho(x)v)-\rho(\mathcal{I}(x))v, \forall x, y \in A.	
 \end{eqnarray}
Then, from the Definition $\ref{De}$, we conclude that $(\eta,v)$ is an antiderivation of $(A,\mathcal{I})$ with values in $(V,\mathcal{T})$ i.e., 
$\mathcal{Z}^{1}_{RB}(A,V)=\mathcal{AD}er(A,V)$. Furthermore, we have $\mathcal{B}^{1}_{RB}(A,V)=\{\delta_{RB}^0(v,v):=(\delta^0(v),-v),v\in V\}=Inn\mathcal{AD}er(A,V)$. Then, 
\begin{equation}
\mathcal{H}^{1}_{RB}(A,V)=\mathcal{AD}er(A,V)/Inn\mathcal{AD}er(A,V).
\end{equation}
 For the particular case  when $V=A$ (the adjoint representation), we obtain
 \begin{equation}
 	\mathcal{H}^{1}_{RB}(A):=\mathcal{AD}er(A)/Inn\mathcal{AD}er(A).
 \end{equation}
\end{itemize}
\begin{example}
From Example \ref{cant}, we deduce that elements of the first cohomology space $\mathcal{H}^{1}_{RB}(A)$ of the $\lambda$-weighted Jacobi-Jordan algebra in Example \ref{F} are the classes of those of the set
$$\left\{\left(\begin{pmatrix}
	a_{11}&0	\\
	0&-2a_{11}	
\end{pmatrix}, \begin{pmatrix}
	\frac{3a_{11}b}{d}	\\
	\\
	0	
\end{pmatrix} \right), a_{11}\in\mathbb{R}\right\}.$$
\end{example}
Let consider again the following example.
 \begin{example}
 	Let $(A,\ast)$ be the three-dimensional Jacobi-Jordan algebra where nonzero products in a given basis $\{e_{1}, e_{2}, e_{3}\}$  are $e_{1}\ast e_{2}=e_{3}=e_{2}\ast e_{1}$ (see \cite{Momo6}).
 	The linear map $\mathcal{I}: A\rightarrow A$ given in the basis $\{e_{1}, e_{2}, e_{3}\}$  by 
 	\begin{eqnarray}
 		\mathcal{I}=		
 		\begin{pmatrix}
 			r_{11}&r_{12}&r_{13}\\
 			r_{21}&r_{22}&r_{23}\\
 			r_{31}&r_{32}&r_{33}	
 		\end{pmatrix}
 		, r_{ij}\in \mathbb{R} \mbox{ for $1\leq i,j\leq 3$,}\nonumber
 	\end{eqnarray}
 	is a $\lambda$-weighted Rota-Baxter operator on 
 	$(A,\ast)$ if and only if
 	\begin{equation}
 		\left\{\begin{array}{ll}
 			r_{13}=r_{23}=0\\
 			r_{11}r_{21}-r_{21}r_{33}=0\\	
 			r_{12}r_{22}-r_{12}r_{33}=0\\
 		(r_{22}+r_{11}+\lambda)r_{33}-	        r_{11}r_{22}-r_{12}r_{21}=0.\\
 		\end{array} \right.\nonumber
 	\end{equation}
 In particular for  $r_{12}\neq 0$ and $r_{21}\neq 0$, we must have 	$r_{12}=\frac{(\lambda+r_{11})r_{11}}{r_{12}}$
 with $r_{11}\neq 0$, $r_{11}\neq -\lambda$ and $r_{33}=r_{22}=r_{11}.$ Hence,
 \begin{eqnarray}
 	\mathcal{I}=		
 	\begin{pmatrix}
 		r_{11}&r_{12}&0\\
 		\frac{(\lambda+r_{11})r_{11}}{r_{12}}&r_{11}&0\\
 		r_{31}&r_{32}&r_{11}	
 	\end{pmatrix}
 	.\nonumber
 \end{eqnarray}
Let $\delta: A\rightarrow A$ be a linear map and $x\in A$. Then $(\delta,x)$ is an antiderivation of $(A,\mathcal{I})$ if and only if $x=\alpha e_1+\beta e_2+x_3 e_3$ where 
\begin{eqnarray}
\alpha&=&-d_{32}+\frac{3d_{11}r_{12}r_{31} }{r_{11}(r_{11}+\lambda)},  \nonumber\\	
\beta&=& -d_{31}+\frac{3d_{11}r_{32}   }{r_{12} }\nonumber
\end{eqnarray}
and the
matrix of $\delta$ on the basis $\{e_1, e_2, e_3\}$ is given by \begin{eqnarray}
	\delta=		
	\begin{pmatrix}
		d_{11}&0&0\\
		0&d_{11}&0\\
		d_{31}&d_{32}&-2d_{11}	
	\end{pmatrix}
	, \mbox{ $d_{ij}\in\mathbb{R}$ for $1\leq i,j\leq 3$.}\nonumber
\end{eqnarray}
In addition, for any $x\in A$, $(L_x,-x)$ where $L_x(y)=x\ast y,\ \forall y \in A$, is an (inner) antiderivation  of $(A,\mathcal{I})$ if and only if $x=x_1e_1+x_2e_2+x_3e_3$ and the matrix of $L_x$ on the basis $\{ e_1,e_2,e_3\}$ is given by
$\begin{pmatrix}
	0&0	&0\\
	0&0& 0\\
	x_2&x_1&0	
\end{pmatrix}$
 where $x_1,x_2,x_3\in\mathbb{R}.$\\
 Observe that the dimension of $\mathcal{AD}er(A)$ (respectively $Inn\mathcal{AD}er(A)$) is  $4$ (respectively $3$). Hence, the dimension of $\mathcal{H}^{1}_{RB}(A)$ is $1$ and its elements can be written as the classes of those of the set
 $$\left\{\left(\begin{pmatrix}
 	d_{11}&0&0\\
 	0&d_{11}&0\\
 	0&0&-2d_{11}	
 \end{pmatrix},\begin{pmatrix}
 \frac{3d_{11}r_{12}r_{31} }{r_{11}(r_{11}+\lambda)}\\
 \\
\frac{3d_{11}r_{32}   }{r_{12} } \\
\\
 0	
\end{pmatrix}\right),\,\ d_{11}\in\mathbb{R}\right\}.$$
 \end{example}
	
\end{document}